\documentclass[11pt]{amsart}
\usepackage{amsfonts}
\usepackage{amsmath}
\usepackage{amsthm}
\usepackage{bm}
\usepackage{bbm}
\usepackage{graphicx}
\usepackage{float}
\usepackage{enumerate}
\usepackage{natbib}

\usepackage{pdfsync}

\usepackage{mathtools}

\numberwithin{equation}{section}

\allowdisplaybreaks

\theoremstyle{plain}
\newtheorem{theorem}{Theorem}[section]
\newtheorem{lemma}[theorem]{Lemma}
\newtheorem{proposition}[theorem]{Proposition}
\newtheorem{corollary}[theorem]{Corollary}
\theoremstyle{definition}
\newtheorem{definition}[theorem]{Definition}
\newtheorem{remark}[theorem]{Remark}
\newtheorem{condition}[theorem]{Condition}
\newtheorem{example}[theorem]{Example}

\newtheorem*{notation}{Notation}

\newcommand{\bx}{{\bf x}}
\newcommand{\by}{{\bf y}}
\newcommand{\ba}{{\bf a}}
\newcommand{\byrec}{\by([\bi:\bj])}

\newcommand{\bs}{{\bf s}}
\newcommand{\bt}{{\bf t}}
\newcommand{\bi}{{\bf i}}
\newcommand{\bj}{{\bf j}}

\newcommand{\bbm}{{\bf m}}
\newcommand{\bn}{{\bf n}}
\newcommand{\bp}{{\bf p}}
\newcommand{\br}{{\bf r}}

\newcommand{\btheta}{{\bm \theta}}

\newcommand{\XRF}{(X(\bt): \bt \in \Z^k)}
\newcommand{\YRF}{(Y(\bt): \bt \in \Z^k)}

\newcommand{\BX}{{\bf X}}
\newcommand{\BXrec}{\BX([\bi:\bj])}
\newcommand{\BXRF}{(\BX(\bt): \bt \in \Z^k )}

\newcommand{\BY}{{\bf Y}}

\newcommand{\BYRF}{(\BY(\bt):\bt\in \Z^k)}

\newcommand{\BZ}{{\bf Z}}

\newcommand{\BT}{{\bf T}}

\newcommand{\BTheta}{{\bf {\Theta}}}
\newcommand{\BThetarec}{\BTheta([\bi:\bj])}
\newcommand{\BThetaRF}{(\BTheta(\bt):\bt \in \Z^k)}
\newcommand{\BThetaRS}{\BTheta^{\text{RS}}}
\newcommand{\BThetaRSRF}{(\BTheta^{\text{RS}}(\bt):\bt \in \Z^k)}

\newcommand{\1}{\mathbbm{1}}
\newcommand{\0}{\mathbf{0}}
\newcommand{\ones}{\mathbf{1}}
\newcommand{\one}{\mathbf{1}}
\newcommand{\binfty}{\bm{\infty}}

\newcommand{\R}{\mathbb{R}}
\newcommand{\Rbar}{\overline{\R}}
\newcommand{\Z}{\mathbb{Z}}
\newcommand{\N}{\mathbb{N}}
\newcommand{\bbS}{\mathbb{S}}

\newcommand{\rec}{\mathcal{R}}
\newcommand{\calI}{\mathcal{I}}
\newcommand{\calBI}{\bm{\mathcal{I}}}

\newcommand{\Prob}{\mathbb{P}}
\newcommand{\E}{\mathbb{E}}

\newcommand{\calL}{\mathcal{L}}
\newcommand{\vague}{\stackrel{v}{\rightarrow}}

\newcommand{\fBm}{\mathrm{fBm}}

\newcommand{\bma}{\begin{matrix*}[r]}
\newcommand{\ema}{\end{matrix*}}

\newcommand{\eid}{\stackrel{d}{=}}
\newcommand{\vep}{\varepsilon}

\begin{document}

\title[]
{Regularly Varying Random Fields}

\author{Lifan Wu}
\address{School of Operations Research and Information Engineering\\
	Cornell University \\
	Ithaca, NY 14853}
\email{lw529@cornell.edu}

\author{Gennady Samorodnitsky}
\address{School of Operations Research and Information Engineering\\
and Department of Statistical Science \\
Cornell University \\
Ithaca, NY 14853}
\email{gs18@cornell.edu}

\thanks{This research was partially supported by the ARO grant W911NF-12-10385 at Cornell 
	University.}
\subjclass{Primary 60G70, 91B72. Secondary 62E20. }
\keywords{regular variation, random field, tail field, spectral field,
 extremal index, Brown-Resnick random field
\vspace{.5ex}}

\begin{abstract}
We study the extremes of multivariate regularly varying random
fields. The crucial tools in our study are the tail field and the
spectral field, notions that extend the  tail and spectral processes
of \cite{basrak:segers:2009}. The spatial context requires multiple
notions of extremal index, and the tail and spectral fields are
applied to clarify these notions and other aspects of extremal
clusters. An important application of the techniques we develop is
to the  Brown-Resnick random fields.  
\end{abstract}

\maketitle

\section{Introduction}

An $\R^d$-valued random
 vector $\BX$ is said to have a multivariate regularly varying
 distribution with exponent $\alpha>0$ if there exists
 a regularly varying with exponent $\alpha$ function $V:\R_{+} \to
 \R_{+}$, and a nonzero Radon measure $\mu$ 
 on $(\Rbar)^d \backslash \{\0\} = [-\infty,\infty]^d\backslash\{\0\}$
 that does not charge infinite points, such that
\begin{align}
\frac{\Prob(x^{-1}\BX \in \cdot)}{V(x)} \vague \mu(\cdot) \label{eq:rv_def}
\end{align}
(vaguely) as $x \to \infty$. The limiting measure $\mu$ is called the
tail measure of $\BX$ and it possesses the scaling property 
$\mu(uA)=u^{-\alpha}\mu(A)$ for any $\alpha>0$ and a measurable set $A
\subset  
(\Rbar)^d\backslash \{\0\}$; see
e.g. \cite{resnick:1987,resnick:2007}. It is usual to say simply that
$\BX$ is regularly varying. 

Infinite-dimensional notions of regular variation are more
complicated, but they have been developed as well. The notion of 
regularly varying stochastic process with sample paths in
$\mathbb{D}([0,1])$ was introduced in \cite{hult:lindskog:2005},  and
it was extended to random fields with sample paths in
$\mathbb{D}([0,1]^d)$ in \cite{davis:mikosch:2008}. 

When stationarity
is present, each observation of the stochastic process is equally
likely to be an extreme, and it is of interest to 
determine how these extremes cluster or, in other words, how these
extremes differ from the extremes of i.i.d. observations with the same
marginal distributions. The extremal index of a stationary process,
introduced by 
\cite{leadbetter:1983},   measures the sizes of extremal clusters. 
 Under the additional assumption of multivariate regular variation,
 \cite{davis:mikosch:2009} introduced the extremogram to capture the
 dependence of the extremes in a stationary regularly varying 
 stochastic process. In order to describe the extremal dependence of
 an entire stochastic process, an unpublished work of
 \cite{owada:samorodnitsky:2014} introduced the notion of a tail
 measure for a regularly varying stochastic process, and better
 known notions are those of the tail and spectral processes developed
 by \cite{basrak:segers:2009}.  

Extending some  of these notions to random fields is challenging due to
the lack of natural order in the time domain. \cite{choi:2002} proved
the existence of a spatial extremal index under the coordinate-wise
mixing condition introduced by \cite{leadbetter:rootzen:1998}, while 
\cite{ferreira:pereira:2008} proposed a way to compute it. 
Recently, \cite{cho:davis:ghosh:2016} formulated the notion of an
extremogram for random fields. In this paper, we extend the theory of
the tail and spectral processes of \cite{basrak:segers:2009} 
 to $\R^d$-valued regularly varying random fields with parameter space
 $\Z^k$. At the same time and independently, a part of this extension was also 
done in \cite{basrak:plannic:2018}, but the goals of that paper are
different. We will mention the similarities in the sequel.

The structure of this paper is as follows. In Section \ref{sec:tail}
we introduce the notion of the tail field corresponding to a
stationary regularly
varying random field. Its properties are studied Section
\ref{sec:spectral}, where the notion of the spectral field is also
introduced. These two notions are analogous to the notions of the tail
and spectral processes of \cite{basrak:segers:2009}. A general
discussion of the possible notions of the spatial extremal index is in
Section \ref{sec:extind}. The point process description of the
extremal clusters is extended from the case of one-dimensional time to
random fields in Section \ref{sec:pp}.  An application to
Brown-Resnick random fields is in Section \ref{sec:brf}. 

\begin{notation}
As usual, letters such as $X$, stand for random variables, while bold
letters, such as $\BX$, stand for random vectors.  Similarly, $\bi$
and $i$ stand for the indices in $\Z^k$ and $\Z$, correspondingly. We use
the notation $\N_0$ for $\N\cup \{0\}$.  For a pair of indices $\bi$
and $\bj$, we say that $\bi \le \bj$ if $i_\ell \le j_\ell$ for all
$\ell = 1,\dots,k$, in which  case $\BXrec$ is the random vector
$(\BX(\bt):\bi \le \bt \le \bj)$. The hypercubes $[(-\bn+ \ones )
:\bn - \ones ]$ and $[\0 : \bn - \ones]$ are denoted by $\rec_\bn$ and
$\rec^+_\bn$, respectively.   

For a random field $\BXRF$ and a
finite set $A \subset \Z^k$ we write $M_X(A)$ for $\max_{\bt \in A}
\|\BX(\bt)\|$. Also, we write $\0$ and $\ones$ for
the vectors of all 0's and 1's, respectively. Finally,  all the vector
operations in this paper are performed element-wise.  
\end{notation}

\medskip
\section{The Tail Field} \label{sec:tail}
Let $\BXRF$ be an $\R^d$-valued random field. It is said to be jointly
regularly  varying if the random vector
$(\BX(\bt_1),\dots,\BX(\bt_n))$ is regularly  varying in $\R^{nd}$ for
any  $\bt_1,\dots,\bt_n\in \Z^k$. The following result is an extension
of Theorem 2.1 in \cite{basrak:segers:2009} to random fields. We will
see that only a partial extension is possible. 

\begin{theorem}\label{thm:tailfield}
An  $\R^d$-valued stationary random field $\BXRF$ is jointly regularly
varying with  index $\alpha>0$ if and only if there exists a random
field $\BYRF$  such that
	\begin{align}\label{eq:lawY}
	\calL\left(x^{-1} \BX(\bt):\bt \in \Z^k \Big| \|\BX(\0)\|>x \right) \to \calL \BYRF
	\end{align}
	as $x \to \infty$ in the sense of convergence of the
        finite-dimensional distributions, and $\Prob(\|\BY(\0)\|>y) = y^{-\alpha}$ for $y \ge 1$. 
\end{theorem}
Extending the terminology of \cite{basrak:segers:2009}, we call the
limiting random field  $\BYRF$ {\it the tail field} of the stationary
field $\BXRF$. 

\begin{proof}[Proof of Theorem \ref{thm:tailfield}] The argument is
  similar to the case of the one-dimensional time. Suppose first that
  $\BXRF$ is jointly regularly varying. 
 Then for arbitrary index pairs $\bi \le \bj$, $\BXrec$ is a regularly
 varying vector with index $\alpha$, By stationarity, the function
 $V$ in \eqref{eq:rv_def} can be chosen to be $V(x) =
 \Prob(\|\BX(\0)\|>x)$ regardless of $\bi,\bj$, so there exists a Radon measure $\mu_{\bi,\bj}$ on $\bigl(\Rbar^d\bigr)^{\prod_{\ell = 1}^k(j_\ell-i_\ell+1)}\backslash \{\0\}$ such that 
	\begin{align}\label{eq:mu}
	\frac{1}{\Prob(\|\BX(\0)\|>x)} \Prob\left(x^{-1} \BXrec \in \cdot\right) \vague \mu_{\bi,\bj}(\cdot)
	\end{align}
as $x\to \infty$. The restriction  $\nu_{\bi,\bj}$ to the set
$\{\byrec \mid \|\by(\0)\|>1 \}$ is, by definition, a probability
measure, and the collection of the probability measures $\bigl(
\nu_{\bi,\bj}\bigr)$ is, clearly, consistent, in the sense that, if
$\bi_2\leq\bi_1\leq \bj_1\leq \bj_2$ then the measure
$\nu_{\bi_1,\bj_1}$ is obtained from the measure
$\nu_{\bi_2,\bj_2}$ by integrating out the redundant dimensions. By
the Kolmogorov extension theorem there is a random field $\BYRF$ whose
finite-dimensional distributions are determined by the family $\bigl(
\nu_{\bi,\bj}\bigr)$. Then \eqref{eq:lawY} follows from
\eqref{eq:mu}, and the Pareto distribution of $\|\BY(\0)\|$ follows as
in the case of the one-dimensional time. 
 
In the opposite direction, suppose that \eqref{eq:lawY} holds for all
$\bi \le \bj \in  \Z^k$, and $\Prob(\|\BY(\0)\|>y) = y^{-\alpha}$. As
in the case of the one-dimensional time, for $y \ge 1$ we have 
\begin{align*}
	\frac{\Prob(\|\BX(\0)\|>xy)}{\Prob(\|\BX(\0)\|>x)}  = \Prob(\|x^{-1}\BX(\0)\|>y \,\vert\, \|\BX(\0)\|>x) \to \Prob(\|\BY(\0)\|>y) = y^{-\alpha}
	\end{align*}
as $x\to\infty$, so that  $\|\BX(\0)\|$ is a regularly varying variable with index
$\alpha$. We need to show  that for any $\bi \le \bj \in \Z^k$,
	\begin{align*}
	\frac{1}{\Prob(\|\BX(\0)\|>x)} \Prob\left(x^{-1}\BXrec \in \cdot \right)
	\end{align*}
converges vaguely as $x \to \infty$, 
and by the already established
regular variation of $\|\BX(\0)\|$, it is enough to establish weak
convergence on the set of vectors for which the norm of $\BX(\bt)$ is at least 1 for some fixed  $\bi \le \bt \le \bj$. We
will, in fact, show weak convergence to the law of the random vector 
$\BY([\bi -\bt:\bj - \bt])$. Indeed, on the relevant
set,  by  stationarity, 
	\begin{align*}
	&\frac{1}{\Prob(\|\BX(\0)\|>x)} \Prob\left(x^{-1}\BXrec \in \cdot \right) \\
	=&\frac{1}{\Prob(\|\BX(\0)\|>x)} \Prob\left(x^{-1}\BXrec \in \cdot \,,\|x^{-1}\BX(\bt)\| > 1\right)\\
	=& \Prob\left(x^{-1}\BXrec \in \cdot \mid\|x^{-1}\BX(\bt)\| > 1\right)\\
=&\Prob\left(x^{-1}\BX([\bi -\bt:\bj - \bt])\in \,\cdot\mid\|\BX(\0)\| > x\right)\\
	\to&\Prob\left(\BY([\bi -\bt:\bj - \bt]) \in \,\cdot\right)
\end{align*}
as $x\to \infty$, as required.   
\end{proof}

\begin{remark}
A similar statement is in Theorem 3.1 of \cite{basrak:plannic:2018}.
When the time is one-dimensional, \cite{basrak:segers:2009} proved that
the weak convergence on the set of nonnegative times, 
	\begin{align} \label{e:nonneg.hf}
	\calL\left(x^{-1}\BX(t):\, t\in\N_0 \Big|\|\BX(0)\|>x\right) \to
          \calL\left(\BY(t):\, t\in\N_0\right)\,,
	\end{align}
sufficed to guarantee the joint regular variation of the original
process. Interestingly, the obvious analogue of this statement for
random fields is false, as the following example of a scalar-valued
random field with 2-dimensional time illustrates.  
\end{remark}

\begin{example} \label{ex:counterex}
Let $ (Z_1,Z_2)$ be a random vector such that $ (Z_1,Z_2)\eid
(Z_2,Z_1)$, $Z_1$ is regularly varying with index $\alpha$, but 
the random vector $(Z_1,Z_2)$ itself is not regularly varying. For completeness, we will construct an example of
such a vector below. 

Let
$ \bigl( Z_1^{(j)},Z_2^{(j)} \bigr)$, 
$j\in\Z$,   be iid copies of $(Z_1,Z_2)$. 	
We define a scalar-valued random field $(X(\bt):\bt \in \Z^2)$ by letting 
	\[X(\bt) = \begin{cases}
	Z_1^{(t_1+t_2)}, &\mbox{ if } t_1 \mbox{ is odd}\\
	Z_2^{(t_1+t_2)}, &\mbox{ if } t_1 \mbox{ is even}\\
	\end{cases}. \]
It is clearly stationary. We claim that 
\begin{equation} \label{e:first.q}
\calL\left(x^{-1} X(\bt):\bt \in \N_0^2 \Big| |X(\0)|>x \right) \to 
\calL\left(Y(\bt):\bt \in \N_0^2 \right)
\end{equation} 
as $x\to\infty$, where $Y(\0)$ has the Pareto$(\alpha)$ distribution,
and $Y(\bt)=0$ for each $\bt\not=\0$. Indeed, since $\0$ is the
only point in $\N_0^2$ on the line $t_1+t_2=0$, $X(\bt)$ is
independent of $X(\0)$ for each $\bt\in
\N_0^2\setminus\{\0\}$. Therefore, for any such $\bt$ we have 
$\calL(x^{-1}X(\bt)| X(\0) > x) \to \delta_\0$ as $x \to
\infty$. Therefore, \eqref{e:first.q}   follows since 
$X(\0)$ is regularly varying with index $\alpha$ because of the assumed
regular variation of $Z_1$.    Therefore,  \eqref{e:first.q} holds. Note that
the latter is the obvious analogue of \eqref{e:nonneg.hf} for a random
field. 

However, the random field $(X(\bt):\bt \in \Z^2)$ is not regularly
varying. To see this, note that with 
$$
 \bt_1=\left(\bma 0\\0\ema\right), \ \bt_2=
 \left( \bma-1\\1\ema \right) 
$$
we have $\bigl( X(\bt_1),X(\bt_2)\bigr)\eid (Z_1,Z_2)$, which, by the
assumption, is not regularly varying. 

It remains to construct a random vector $ (Z_1,Z_2)$ such that $ (Z_1,Z_2)\eid
(Z_2,Z_1)$, $Z_1$ is regularly varying with index $\alpha$, but 
the random vector $(Z_1,Z_2)$ itself is not regularly varying. Let
$a_n=n!, \, n=1,2,\ldots$. Let $Z\geq 1$ have the standard
Pareto$(\alpha)$ distribution. If $Z\in [a_{2n-1},a_{2n})$ for some
$n=1,2,\ldots$, set $Z_1=Z_2=Z$. If $Z\in [a_{2n},a_{2n+1})$ for some
$n=1,2,\ldots$, take $Z_1$ and $Z_2$ be standard
Pareto$(\alpha)$ random variables conditioned on being in the interval 
$[a_{2n},a_{2n+1})$ but otherwise independent. Formally, for any
two-dimensional Borel set $A$,
\begin{align*}
\Prob\bigl( (Z_1,Z_2)\in A\bigr) &=
\sum_{n=1}^\infty \int_{a_{2n-1}}^{a_{2n}} \1\bigl( (z,z)\in
                               A\bigr)az^{-(\alpha+1)}\, dz \\
                                         &+
\sum_{n=1}^\infty \frac{1}{a_{2n}^{-\alpha}-a_{2n+1}^{-\alpha}} 
\int_{a_{2n}}^{a_{2n+1}} \int_{a_{2n}}^{a_{2n+1}} \1\bigl( (z_1,z_2)\in
                               A\bigr)
                                           az_1^{-(\alpha+1)}az_2^{-(\alpha+1)}\,
                                           dz_1\, dz_2\,.
\end{align*}
By construction, $ (Z_1,Z_2)\eid (Z_2,Z_1)$, and each coordinate of
the random vector has the standard Pareto$(\alpha)$ distribution. It
remains to show that the random vector $(Z_1,Z_2)$  is not regularly
varying. Note that
\begin{align*}
\Prob\left( a_{2n-1}^{-1}(Z_1,Z_2)\in (1,2]\times
  (1,2]\right)
\sim \Prob\left( Z\in (a_{2n-1},2a_{2n-1}]\right) \sim
  (1-2^{-\alpha})a_{2n-1}^{-\alpha} 
\end{align*}
as $n\to\infty$. On the other hand,
\begin{align*}
\Prob\left( a_{2n}^{-1}(Z_1,Z_2)\in (1,2]\times
  (1, 2]\right) &\leq \frac{1}{a_{2n}^{-\alpha}-a_{2n+1}^{-\alpha}} 
\left( \int_{a_{2n}}^{2a_{2n}} az^{-(\alpha+1)}\, dz\right)^2 \\
&+\Prob\left(Z\geq a_{2n+1}\right)\sim (1-2^{-\alpha})^2a_{2n}^{-\alpha}
\end{align*}
as $n\to\infty$. Therefore, \eqref{eq:rv_def} cannot hold.

\end{example}

\medskip
\section{Properties of the tail field} \label{sec:spectral}
This section describes the properties of the tail field introduced in
the previous section. These are similar, but not identical, to the
properties of the tail process. In particular, we introduce an object 
parallel to that of the spectral process of \cite{basrak:segers:2009},
which we call the spectral field. The latter is defined as
$\BTheta(\bt) = \BY(\bt) / \| \BY(\0)\|$, $\bt\in \Z^k$, where $\BYRF$
is the tail field of an $\R^d$-valued stationary random field $\BXRF$
that is jointly regularly varying with index $\alpha>0$.  As in the
one-dimensional case, it is easy to check that 
\begin{equation} \label{e:decomp}
\text{the spectral field is independent of} \ \| \BY(\0)\|\,.
\end{equation}

The following proposition can be proved in the same way as for the
one-dimensional time, so we do not include the proof. See also Theorem
3.1 in \cite{basrak:plannic:2018}. Note,
however, that a part of Corollary 3.2 in \cite{basrak:segers:2009}
fails in the case of random fields; see Example \ref{ex:counterex}. 
\begin{proposition} \label{pr:decomp}
	Let $\BXRF$ be an $\R^d$-valued stationary random field, and $\|\BX(\0)\|$ be a regularly varying variable with index $\alpha$ for some $\alpha \in (0,\infty)$. Then $\BXRF$ is jointly regularly varying with index $\alpha$ if and only if there exists a random field $\BThetaRF$ such that
	\begin{align}\label{eq:lawtheta}
	\calL\left(\frac{\BX(\bt)}{\|\BX(\0)\|} : \bt \in \Z^k\,\middle|\,  \|\BX(\0)\|>x \right) \to \calL\BThetaRF
	\end{align}
	as $x \to \infty$.
\end{proposition}

Even though neither the tail field nor the spectral field is
generally stationary, the stationarity of the original random field
$\BXRF$ makes itself felt in the former fields. In particular, it 
leads to a ``change-of-time'' property for these fields. A
similar result in the case of one-dimensional time is a part of
Theorem 3.1 in \cite{basrak:segers:2009}.  We present this property in
a somewhat more general form. 
\begin{theorem} \label{thm:shift}
Let $\BYRF$ be the tail field corresponding to an
$\R^d$-valued stationary random field $\BXRF$ 
that is jointly regularly varying with index $\alpha>0$, and let
$\BThetaRF$ be the corresponding spectral field.  Let $g:
\bigl(\R^d)^{\Z^k} \to \R$ be a bounded measurable function. Take
any $\bs \in \Z^k$. Then the following identities hold: 
\begin{align}
		\E\bigl[ g(\BY(\cdot-\bs))\1\bigl(\BY(-\bs)\not=\0\bigr)\bigr] &= \int_0^\infty
                \E[g(r\BTheta(\cdot))\1(r\|\BTheta(\bs)\|>1)]\,d(-r^{-\alpha})\,, 
\label{e:shift.1a} \\ 
		\E\bigl[g(\BTheta(\cdot-\bs)) \1\bigl(\BTheta(-\bs)\not=\0\bigr)\bigr] &= 
\E\left[g\left(\frac{\BTheta(\cdot)}{\|\BTheta(\bs)\|}\right)
\|\BTheta(\bs)\|^{\alpha}\right]\,. \label{e:shift.2a} 
\end{align}		
 \end{theorem}

\begin{proof}
Since a probability measure on $\bigl(\R^d)^{\Z^k}$ is uniquely
determined by its finite-dimensional distributions, for
\eqref{e:shift.1a} it is enough to
prove that for any $\bi\leq \bj \in \Z^k$ and any bounded measurable
function $g:
\bigl(\R^d\bigr)^{\prod_{\ell = 1}^k(j_\ell-i_\ell+1)} \to
\R$, we have 
\begin{align}
		\E\bigl[ g(\BY([\bi-\bs :
                  \bj-\bs]))\1\bigl(\BY(-\bs)\not=\0\bigr)\bigr] &= \int_0^\infty
                \E[g(r\BThetarec)\1(r\|\BTheta(\bs)\|>1)]\,d(-r^{-\alpha}) \,. \label{e:shift.1}
		\end{align}
Suppose first that $g$ is bounded and continuous. Let $\vep>0$. 
By \eqref{eq:lawY} and  stationarity, the argument of 
\cite{basrak:segers:2009} gives us 
		\begin{align*}
		\E\bigl[ g(\BY([\bi-\bs :
                  \bj-\bs]))\1\bigl(\|\BY(-\bs)\|>\vep\bigr)\bigr] &= \int_\vep^\infty
                \E[g(r\BThetarec)\1(r\|\BTheta(\bs)\|>1)]\,d(-r^{-\alpha})\,.
		\end{align*}
If $g$ is, in addition, nonnegative, then we can let $\vep\downarrow
0$ in this relation, so that monotone convergence theorem gives us
\eqref{e:shift.1} for nonnegative bounded and continuous $g$. The
assumption of nonnegativity can now be removed by writing $g$ as the
difference of its positive and negative parts. Since integrals of
bounded continuous functions uniquely determine a finite measure,
we see that \eqref{e:shift.1} holds without the assumption of
continuity. As in \cite{basrak:segers:2009}, \eqref{e:shift.2a}
follows from \eqref{e:shift.1a} by defining a new bounded measurable
function on $\bigl(\R^d)^{\Z^k}$ as 
$\tilde g(\by)= g(\by/\|\by(\bs)\|)\1(\by(\bs)\not=0)$ and applying 
\eqref{e:shift.1a} to this function. 
\end{proof}

If $h$ is a bounded measurable function on the unit sphere
$\bbS^{d-1}$ in $\R^d$, then choosing $g(\by) =h(\by(\bs))$ if
$\|\by(\bs)\|=1$ and $g(\by) =0$ otherwise produces a bounded
measurable function on $\bigl(\Rbar^d)^{\Z^k}$. Applying
\eqref{e:shift.2a} to this function gives us the identity 
	\begin{align} \label{e:bs.1}
	\mathbb{E}[h(\BTheta(\bs)/\|\BTheta(\bs)\|)\|\BTheta(\bs)\|^\alpha]
          = \mathbb{E}[h(\BTheta(\0))\1(\BTheta(-\bs)\ne\0)]\,. 
	\end{align}
The value of $\E[\|\BTheta(\bs)\|^{\alpha}]$ is a measure of the
effect of changing the ``origin'' of the spectral field from $\0$ to
$\bs$ (recall that $\|\BTheta(\0)\|=1$ by the definition). 
With $h(\cdot) \equiv 1$, \eqref{e:bs.1} reduces to
$\E[\|\BTheta(\bs)\|^{\alpha}] = \Prob(\BTheta(-\bs)\ne \0)$ (so
$\E[\|\BTheta(\bs)\|^{\alpha}]\leq 1$). In particular, for $\delta>0$, 
\begin{align*}
 \lim_{x \to \infty} \Prob\bigl(\|\BX(\0)\|> \delta x \,\vert\,
  \|\BX(\bs)\|>x\bigr) &= \lim_{x \to \infty} \Prob\bigl(\|\BX(-\bs)\|>
  \delta x \,\vert\,   \|\BX(\0)\|>x\bigr) \\
& = \Prob(\|\BY(-\bs)\|>\delta)\,,
\end{align*}
so
\begin{align*}
\E[\|\BTheta(\bs)\|^{\alpha}] &= \Prob(\BTheta(-\bs)\ne \0) \\
&= \Prob(\BY(-\bs)\ne \0) = \lim_{\delta\downarrow 0}
\Prob(\|\BY(-\bs)\|>\delta) \\
&= \lim_{\delta\downarrow 0}\lim_{x \to \infty}
  \Prob\bigl(\|\BX(\0)\|> \delta x \,\vert\, \|\BX(\bs)\|>x\bigr) \,,
\end{align*}
thus providing an intuitive interpretation of the quantity
$\E[\|\BTheta(\bs)\|^{\alpha}]$. Furthermore, assuming that
$\E[\|\BTheta(\bs)\|^{\alpha}]>0$,  \eqref{e:bs.1} says that the two 
probability measures on $\bbS^{d-1}$, 
\begin{align*}
\Prob_1(\cdot) & =\frac{1}{\E[\|\BTheta(\bs)\|^{\alpha}]}  
\E[\|\BTheta(\bs)\|^{\alpha}\1(\BTheta(\bs)/\|\BTheta(\bs)\| \in \cdot)]\,,\\
 \Prob_2(\cdot) & =\frac{\Prob(\BTheta(\0)\in \cdot,\, \BTheta(-\bs)\ne \0)
}{ \Prob(\BTheta(-\bs)\ne \0)}\,,                  
\end{align*}
are equal. Therefore, a necessary and sufficient condition for
$\E[\|\BTheta(\bs)\|^{\alpha}] =1$ is 
$$
\Prob(\BTheta(\0)\in \cdot) =
\E[\|\BTheta(\bs)\|^{\alpha}\1(\BTheta(\bs)/\|\BTheta(\bs)\| \in
\cdot)]\,. 
$$

The above discussion is an extension of the ideas in
\cite{basrak:segers:2009} in the case of on-dimensional time to random fields. 

\medskip

The important ``change-of-time'' property \eqref{e:shift.2a} has
recently been shown in \cite{janssen:2018} to be equivalent, in the
case of the one-dimensional time, to a certain distributional
invariance property of the spectral process. As we explain below, this
equivalence extends to random fields. We start with a simple extension
of Lemma 2.2 {\it ibid.} It describes a rather unexpected property of the
spectral field. The argument requires the notion of invariant
order. A complete  order $\prec$ on $\Z^k$ is called invariant if
	$\bs \prec \bt$ for $\bs,\bt\in \Z^k$ implies that $\bs +\bi \prec
	\bt+\bi$ for any $\bi\in\Z^k$. An example of an invariant order is
	the lexicographic (or dictionary) order: for $\bs,\bt \in \Z^k$, we
	say that $\bs \prec \bt$ if either (1) $s_1 < t_1$, or (2) there exists $2
	\le \ell \le k$ such that $s_i = t_i$ for all $i = 1,\dots,\ell-1$,
	and $s_\ell < t_\ell$.

\begin{lemma} \label{l:anja} Let $\BThetaRF$ be an $\R^d$-valued random field such
  that $\Prob(\|\BTheta(\0)\| = 1) = 1$ and satisfies
  \eqref{e:shift.2a}.  Then $\|\BTheta(\bt)\| \to 0$ a.s. as $\|\bt\|
  \to \infty$ if and only if 
 $\sum_{\bt \in \Z^k} \|\BTheta(\bt)\|^{\alpha} < \infty$ a.s. 
\end{lemma}
\begin{proof}
Trivially,  the summability condition implies that the values of the
field vanish at infinity. In the other direction, fix an invariant
order on $\Z^k$, and 
suppose that the event $\{ \|\BTheta(\bt)\| \to 0$ as $\|\bt\|
  \to \infty\}$ has probability 1. On this event there is, clearly, a
  finite number of points in $\Z^k$ over which
  $\|\BTheta(\bi)\|$ achieves the supremum
  $\sup_{\bt \in \Z^k}\| \BTheta(\bt)\|$. Therefore, on
  this event we can define a $\Z^k$-valued random variable $\BT$ such
  that $\|\BTheta(\BT)\|=\sup_{\bt \in \Z^k}\| \BTheta(\bt)\|$ and any other point of $\Z^k$ with this property succeeds
  $\BT$ in the invariant order. If, to the contrary, we have 
$\Prob(\sum_{\bt \in \Z^k} \|\BTheta(\bt)\|^{\alpha} = \infty)>0$,
then there exists  $\bi \in \Z^k$ such that $P(\sum_{\bt \in \Z^k}
\|\BTheta(\bt)\|^{\alpha}  = \infty,\BT=\bi)>0$, which gives us
	\begin{align*}
	\infty =&\E\left[\sum_{\bt \in \Z^k} \|\BTheta(\bt)\|^{\alpha}\1(\BT = \bi)\right] = \sum_{\bt \in \Z^k} \E\left[\|\BTheta(\bt)\|^{\alpha}\1(\BT = \bi)\right].
	\end{align*}
For each $\bi \in \Z^k$ we define a function $g_\bi: \bigl(\Rbar^d)^{\Z^k} \to \R$ as
        follows. If $(\btheta(\bs), \, \bs\in \Z^k)$ is such that
$$
\|\btheta(\bj)\| <\|\btheta(\bi)\| \ \text{for} \ \bj \prec \bi, \, \|\btheta(\bj)\|
\le \|\btheta(\bi)\| \ \text{for} \ \bj \succeq \bi\,,
$$
then we set $g_\bi(\btheta(\bs), \, \bs\in \Z^k)=1$. Otherwise we set
$g_\bi(\btheta(\bs), \, \bs\in \Z^k)=0$.  Clearly, each $g_\bi$ is a
bounded measurable function. Then by the ``change of time property''
\eqref{e:shift.2a}, 
	\begin{align*}
	\infty =& \sum_{\bt \in \Z^k}
          \E\left[\|\BTheta(\bt)\|^{\alpha}g_{\bi}(\BTheta(\bs):\bs \in
          \Z^k)\right] =  \sum_{\bt \in \Z^k}
          \E\left[\|\BTheta(\bt)\|^{\alpha}g_{\bi}\left(\frac{\BTheta(\bs)}{\|\BTheta(\bt)\|}:\bs
          \in \Z^k\right)\right] \\
	=& \sum_{\bt \in \Z^k} \E\left[g_{\bi}\left(\BTheta(\bs - \bt):\bs \in \Z^k\right)\1(\BTheta(-\bt) \ne \0)\right]\\
	=& \sum_{\bt \in \Z^k}\E\left[ \1(\|\BTheta(\bj)\|
           <\|\BTheta(\bi-\bt)\|,\, \bj \prec \bi-\bt)
           \1(\|\BTheta(\bj)\| \le \|\BTheta(\bi-\bt)\|,\, \bj \succeq \bi-\bt ) \1(\BTheta(-\bt) \ne \0)\right]\\
	\le& \sum_{\bt \in \Z^k}\E\left[ \1(\|\BTheta(\bj)\|
           <\|\BTheta(\bi-\bt)\|,\, \bj \prec \bi-\bt)
           \1(\|\BTheta(\bj)\| \le \|\BTheta(\bi-\bt)\|,\, \bj \succeq \bi-\bt )\right]\\
	=&  \sum_{\bt \in \Z^k} P(\BT = \bi - \bt) = 1,
	\end{align*}
	which leads to a contradiction. Hence, $\sum_{\bt \in \Z^k}
        \|\BTheta(\bt)\|^{\alpha} < \infty$ a.s.. 
\end{proof}

As in the one-dimensional case, the spectral field vanishes a.s. at
infinity under Condition \ref{cond:vanish} below; see Theorem
\ref{thm:cluster}. Lemma \ref{l:anja} shows that  under Condition
\ref{cond:vanish} the spectral process also satisfies the stronger
summability statement.

The next theorem is a version of Theorem 2.4 in \cite{janssen:2018}
for random fields. It establishes a certain invaraince property of the
law of a spectral fields satsifying the equaivalent conditions of
Lemma \ref{l:anja}. 
\begin{theorem} \label{t:tanja}
	Let $\BThetaRF$ be an $\R^d$-valued random field such that
$0<\sum_{\bt \in \Z^k} \|\BTheta(\bt)\|^{\alpha} < \infty$ a.s.. Let 
$\calBI$ be an $\Z^k$-valued random element such that
	\begin{align}
	\Prob(\calBI  = \bi \mid \BThetaRF) = \frac{\|\BTheta(\bi)\|^{\alpha}}{\sum_{\bt \in \Z^k}\|\BTheta(\bt)\|^{\alpha}}. \label{eq: indexpmf}
	\end{align}
	for $\bi \in \Z^k$.
	Define 
	\begin{align*}
  \BTheta^{RS}(\bt) =  \frac{\BTheta(\bt
          +\calBI)}{\|\BTheta(\calBI)\|}, \  \bt \in \Z^k\,.
	\end{align*}
	Then a necessary and sufficient condition for the equality of
        the laws 
	\begin{align}
	\calL(\BThetaRSRF ) = \calL(\BThetaRF) \label{eq: RSlaw}
	\end{align}
is that  $\BThetaRF$ satisfies \eqref{e:shift.2a} and
$\Prob(\|\BTheta(\0)\|=1)=1$. 
\end{theorem}
\begin{proof}
	Suppose first \eqref{eq: RSlaw} holds. Since
        $\|\BThetaRS(\0)\|=1$, we must have $\Prob(\|\BTheta(\0)\|=1)=1$. Now, let $g:
	\bigl(\R^d)^{\Z^k} \to \R$ be a bounded measurable function.
        Denoting $\|\BTheta\|_{\alpha} = \left(\sum_{\bt \in \Z^k}
          \|\BTheta(\bt)\|^{\alpha}\right)^{1/\alpha}$, we have 
	\begin{align*}
	&\E[g(\BTheta(\cdot - \bs))\1(\BTheta(-\bs)\ne \0)] 
 	=  \E\left[   g(\BThetaRS(\cdot -
            \bs))\1(\BThetaRS(-\bs)\ne  \0)\right] \\
	= & \E\left[\E\left[ g\left(\frac{\BTheta(\cdot - \bs+\calBI)}{\|\BTheta(\calBI)\|}\right)\1\left(\frac{\BTheta(-\bs+\calBI)}{\|\BTheta(\calBI)\|}\ne \0\right) \,\middle|\,\BThetaRF \right]  \right]\\
 	= & \E\left[\sum_{\bi \in \Z^k} \frac{\|\BTheta(\bi)\|^{\alpha}}{\|\BTheta\|_{\alpha}^{\alpha}} g\left(\frac{\BTheta(\cdot - \bs+\bi)}{\|\BTheta(\bi)\|}\right)\1\left(\BTheta(-\bs+\bi)\ne \0\right) \right].
	\end{align*}
	Similarly,
	\begin{align*}
	&\E\left[g\left(\frac{\BTheta(\cdot)}{\|\BTheta(\bs)\|}\right)
	\|\BTheta(\bs)\|^{\alpha}\right] \\
	= &\E\left[g\left(\frac{\BThetaRS(\cdot)}{\|\BThetaRS(\bs)\|}\right)
	\|\BThetaRS(\bs)\|^{\alpha}\right]\\
	= &\E\left[\sum_{\bj \in \Z^k} \frac{\|\BTheta(\bj)\|^{\alpha}}{\|\BTheta\|_{\alpha}^{\alpha}} g\left(\frac{\BTheta(\cdot+\bj)}{\|\BTheta(\bs+\bj)\|}\right)
	\left\|\frac{\BTheta(\bs+\bj)}{\|\BTheta(\bj)\|}\right\|^{\alpha} \1(\BTheta(\bj) \ne \0)\right]\\
	= &\E\left[\sum_{\bi \in \Z^k}
            \frac{\|\BTheta(\bi)\|^{\alpha}}{\|\BTheta\|_{\alpha}^{\alpha}}
            g\left(\frac{\BTheta(\cdot-\bs+\bi)}{\|\BTheta(\bi)\|}\right)\1(\BTheta(-s+\bi)
            \ne \0) \right]
	\end{align*}
	by substituting $\bj=\bi - \bs$. The equal results of
        these two calculations show that the random field has the
        property \eqref{e:shift.2a}. 
 
 In the other direction, suppose that the random field $\BThetaRF$
 satisfies \eqref{e:shift.2a} and $\Prob(\|\BTheta(\0)\|=1) = 1$. For
 any bounded measurable function $g$,  
	\begin{align*}
	\E[g\BThetaRSRF] = \sum_{\bi \in \Z^k} \E\left[g\left(\frac{\BTheta(\bt - \bi)}{\|\BTheta(-\bi)\|} \right) \frac{\|\BTheta(-\bi)\|^{\alpha}}{\|\BTheta\|_{\alpha}^{\alpha}}\1(\BTheta(-\bi) \ne \0)\right].
	\end{align*}
	Define a new function $\bar{g}$ by 
	\begin{align*}
	\bar{g}\left( \btheta(\bt): \, \bt\in \Z^k\right) = g\left(\frac{\btheta(\bt)}{\|\btheta(\0)\|}: \bt \in \Z^k \right)\frac{\|\btheta(\0)\|^{\alpha}}{\|\btheta\|_{\alpha}^{\alpha}}
	\end{align*}
if $\btheta(\0)\not=\0$ and $\|\btheta\|_{\alpha}^{\alpha}=\sum_{\bi \in \Z^k}
\|\btheta(\bi)\|^\alpha<\infty$. If these conditions do not hold, set
$\bar g=0$. Since $\bar g$ is a bounded and measurable function, we
have by \eqref{e:shift.2a} and the fact that
$\Prob(\|\BTheta(\0)\|=1)=1$, 
	\begin{align*}
	\E[g\BThetaRSRF] =& \sum_{\bi \in \Z^k} \E\left[\bar{g}(\BTheta(\bt - \bi):\bt \in \Z^k)\1(\BTheta(-\bi) \ne \0)\right]\\
	=& \sum_{\bi \in \Z^k} \E\left[\bar{g}\left(\frac{\BTheta(\bt)}{\|\BTheta(\bi)\|}:\bt \in \Z^k\right)\|\BTheta(\bi)\|^{\alpha}\right]\\
 =& \sum_{\bi \in \Z^k} \E\left[g(\BTheta(\bt):\bt \in \Z^k) \frac{\|\BTheta(\bi)\|^{\alpha}}{\|\BTheta\|_{\alpha}^{\alpha}}\right]\\
	= &\E\left[g(\BTheta(\bt):\bt \in \Z^k)\right],
	\end{align*} 
	proving \eqref{eq: RSlaw}. 
\end{proof}

\medskip
\section{Extremal index of a random field} \label{sec:extind}

The extremal index is one of the major ways to characterize how the
extremes of a stationary sequence cluster; it was introduced in
\cite{leadbetter:1983} and extensively studied and used ever
since. The corresponding notion for random fields appeared in
\cite{ferreira:pereira:2008}. One of the attractive features of the
notion of the extremal index is that it admits multiple
interpretations. These different points of view on extremal index,
however, turn out to be equivalent only under appropriate technical
conditions (and the equivalences turn out to be even more strained for
random fields).  In fact, the original definition of the extremal index
itself includes an assumption of its existence. For jointly regularly
varying random fields, the tail field sheds new light on the notion of
the extremal index. Importantly, no assumptions of existence are
required for the tail field-based notions of the extremal index
(apart, of course, from the regular 
variation). In order to clarify the situation, we keep the definitions distinct. 

\begin{definition}  \label{def:classical.ei}
An $\R^d$-valued stationary random
  field $\BXRF$ has a {\bf classical extremal index} $\theta_{\rm cl}$ if for
  each $\tau>0$   and any array $(u_{\bn}(\tau))$ satisfying 
	\begin{align}
	\left(\prod_{\ell = 1}^k n_\ell \right)
          \Prob(\|\BX(\0)\|>u_{\bn}(\tau)) \to \tau 
\label{eq:u}
	\end{align}
	as $\bn \to \binfty$, it also holds that
	\begin{align} \label{e:u1}
	\Prob\left(M_X(\rec^+_{\bn}) \le u_{\bn}(\tau)\right) \to
          e^{-\theta_{\rm cl} \tau}.
	\end{align}
\end{definition}

\begin{remark}
It is common to formulate the definition of the classical extremal
index by requiring that \eqref{eq:u} and \eqref{e:u1} hold for some
array $(u_{\bn}(\tau))$. This appears to tie the notion to a
particular choice of the array, and does not seem to broaden the
applicability of the definition. 
\end{remark}

\begin{definition}  \label{def:block.ei}
An $\R^d$-valued stationary random
  field $\BXRF$ has a {\bf block extremal index} $\theta_{\rm b}$ if
  for some 
  array $(\br_{\bn})$ increasing to $\bm \infty$ such that $\br_\bn
  /\bn \to \0$, for   each $\tau>0$   and any array $(u_{\bn}(\tau))$ satisfying 
\eqref{eq:u}, it holds that 
	\begin{align} \label{eq:theta}
	\theta_{\rm b}  = \lim_{\bn \to \binfty}
          \frac{\Prob(M_X(\rec^+_{\br_{\bn}})>u_{\bn}(\tau))}{\left(\prod_{\ell
          = 1}^k r_{n_\ell}\right)\Prob(\|\BX(\0)\|>u_{\bn}(\tau))} 
	\end{align}
\end{definition}

Under certain conditions the block extremal index coincides with the
classical extremal index, assuming the latter exists. One such set of
conditions is the so called coordinatewise tail mixing condition; see
Proposition 3.2 in \cite{pereira:martins:ferreira:2017}. 

The next definition of the extremal index is well known in the case of
the one-dimensional time, but does not seem to have been formulated
for random fields. It concentrates on the conditional probability of
the random field being free of exceedances over the rest of a
hypercube given an exceedance at one of the corners of the hypercube. 

\begin{definition}  \label{def:run.ei}
An $\R^d$-valued stationary random
  field $\BXRF$ has a {\bf run extremal index} $\theta_{\rm run,\bi}$ with respect
  to $\bi\in \{0,1\}^k$ if for some array $(\br_{\bn})$
increasing to $\bm \infty$ such that $\br_\bn   /\bn \to \0$, and any
array $(u_{\bn}(\tau))$ satisfying  
\eqref{eq:u} for some $\tau>0$, it holds
that 
\begin{align}
\theta_{\rm run,\bi}=\lim_{\bn \to \binfty}\Prob(M_X(\rec^+_{\br_{\bn}}\backslash
  \{\bt_{\bn,\bi}\}) \le u_{\bn}(\tau) \mid X(\bt_{\bn,\bi}) >
  u_{\bn}(\tau)) \,,
\end{align} 
where $(t_{\bn,\bi})_l=r_{n_l}-1$ if $i_{n_l}=1$ and $0$ if
$i_{n_l}=0$. 
\end{definition}

When the time is one-dimensional, the hypercube has two corners, and
the stationarity implies that the run extremal index, if it exists, is
the same for the two corners. Indeed,  
\begin{align*}
&\Prob\left(\max_{t=1,\dots,r_n-1} X(t) \le u_n(\tau)\mid X(0) > u_n(\tau)\right) \\
=&\frac{\Prob\left(\max_{t=1,\dots,r_n-1} X(t) \le u_n(\tau),  X(0) >
   u_n(\tau)\right)}{\Prob(X(0)>u_n(\tau))}\\ 
=&\frac{\Prob\left(\max_{t=0,\dots,r_n-2} X(t) \le u_n(\tau)\right)-\Prob\left(\max_{t=0,\dots,r_n-1} X(t) \le u_n(\tau)\right)}{\Prob(X(r_n-1)>u_n(\tau))}\\
=&\frac{\Prob\left(\max_{t=0,\dots,r_n-2} X(t) \le u_n(\tau) , X(r_n-1) > u_n(\tau)\right)}{\Prob(X(r_n-1)>u_n(\tau))}\\
=&\Prob\left(\max_{t=0,\dots,r_n-2} X(t) \le u_n(\tau) \,\middle|\,  X(r_n-1) > u_n(\tau)\right).
\end{align*}
This, however, is no longer necessarily the case that for random
fields the run extremal index is independent of the corner of the
hypercube used to define it, as will be seen in Example
\ref{ex:corners} below. When the time is one-dimensional, under
certain conditions the run extremal index coincides with the 
classical extremal index; one such set of conditions being the AIM
conditions of  \cite{obrien:1987}. As the previous discussion and 
Example \ref{ex:corners} indicate, this is no longer the case for random
fields. 

\begin{definition}  \label{def:tf.ei}
Let $\BXRF$ be a stationary jointly regularly varying $\R^d$-valued
random field $\BXRF$ with the tail field $\BYRF$. Its {\bf tail field
  extremal index} $\theta_{\rm tf,\bi}$ with respect to $\bi\in
\{0,1\}^k$ is 
$$
\theta_{\rm tf,\bi} = \Prob\left( \sup_{\bt:\, \bt (\ones-2\bi)\geq
    \0, \,     \bt\not=\0} \|\BY(\bt)\|\leq 1\right) \,.
$$
\end{definition}
Under appropriate conditions, similar to those of
\cite{basrak:segers:2009} in the one-dimensional time case, the tail
field extremal index coincides  with the run extremal
index and, in particular, the latter exists. 
 
\begin{proposition} \label{pr:corner.tail}
Let $\BXRF$ be a stationary jointly regularly varying random field 
with the tail field $\BYRF$. Let $\bi\in \{0,1\}^k$. Suppose that for
any array $(u_{\bn}(\tau))$ satisfying  
\eqref{eq:u} for some $\tau>0$ and some array $(\br_{\bn})$
increasing to $\bm \infty$ such that $\br_\bn   /\bn \to \0$,
$$
\lim_{M\to\infty} \limsup_{\bn\to\binfty} \Prob\left(
  M_X(\rec^+_{A_{M,\bi}}\backslash
  \{t_{\bn,\bi}\} ) \leq   u_{\bn}(\tau), \,  M_X(\rec^+_{\br_{\bn}}\backslash
  \rec^+_{A_{M,\bi}} )> u_{\bn}(\tau) \bigg|   \|\BX(\0)\|>u_\bn(\tau)\right) =0\,,
$$
where $A_{M,\bi}=\{ \bx\in \rec^+_{\br_{\bn}}:\,
|x_l-(t_{\bn,\bi})_l|\leq M, \, l=1,\ldots, d\}$. Then the run
extremal index $\theta_{\rm run,\bi}$ exists and is equal to the tail
field extremal index $\theta_{\rm tf,\bi}$. 
\end{proposition} 
\begin{proof}
It is enough to consider the case $\bi=\0$, in which case the
condition in the proposition reduces to 
\begin{equation} \label{e:corners.tf.2}
\lim_{M\to\infty} \limsup_{\bn\to\binfty} \Prob\left( M_X(\rec^+_{M\ones}\backslash
  \{\0\} ) \leq   u_{\bn}(\tau), \,  M_X(\rec^+_{\br_{\bn}}\backslash
  \rec^+_{M\ones} )> u_{\bn}(\tau) \bigg|   \|\BX(\0)\|>u_\bn(\tau)\right) =0\,.
\end{equation}
We have for any $M=1,2,\ldots$, any $(\br_{\bn})$
increasing to $\bm \infty$, 
\begin{align*}
  &\Prob\left(\sup_{\0\leq \bt \leq M\ones, \, \bt\not=\0
          }\|\BY(\bt)\|\le 1\right) \\
=\lim_{\bn \to \binfty}& \Prob\left( \sup_{\0\leq \bt \leq M\ones, \,
  \bt\not=\0 }\frac{1}{u_\bn(\tau)} \|\BX(\bt)\|\leq   1\bigg|
  \|\BX(\0)\|>u_\bn(\tau)\right) \\
\geq \limsup_{\bn \to \binfty}\, & \Prob(M_X(\rec^+_{\br_{\bn}}\backslash
  \{\0\} )\le u_{\bn}(\tau) \mid  \|\BX(\0)\| >
  u_{\bn}(\tau)) \,.
\end{align*}
Letting $M\to\infty$, we obtain
\begin{equation} \label{e:corners.tf.1}
\Prob\left(\sup_{\bt \geq\0, \, \bt\not=\0
          }\|\BY(\bt)\|\le 1\right) \geq
\limsup_{\bn \to \binfty}\, \Prob(M_X(\rec^+_{\br_{\bn}}\backslash
  \{\0\} )\le u_{\bn}(\tau) \mid  \|\BX(\0) \|>
  u_{\bn}(\tau)) \,.
\end{equation}
Furthermore,  we can write for $\bn$ large enough, 
\begin{align*}
&\Prob(M_X(\rec^+_{\br_{\bn}}\backslash
  \{\0\} )\le u_{\bn}(\tau) \mid  \|\BX(\0)\| >
  u_{\bn}(\tau)) \\
= & \Prob\left( \frac{1}{u_\bn(\tau)}M_X(\rec^+_{M\ones}\backslash
  \{\0\} ) \leq   1\bigg|
  \|\BX(\0)\|>u_\bn(\tau)\right) \\
- & \Prob\left( M_X(\rec^+_{M\ones}\backslash
  \{\0\} ) \leq   u_{\bn}(\tau), \, M_X(\rec^+_{\br_{\bn}}\backslash
  \{\0\} )> u_{\bn}(\tau) \bigg|   \|\BX(\0)\|>u_\bn(\tau)\right)\,.
\end{align*}
By \eqref{e:corners.tf.2}, letting first $\bn\to\binfty$ and then
$M\to\infty$ gives us  
$$
\Prob\left(\sup_{\bt \geq\0, \, \bt\not=\0
          }\|\BY(\bt)\|\le 1\right) \leq
\liminf_{\bn \to \binfty}\, \Prob(M_X(\rec^+_{\br_{\bn}}\backslash
  \{\0\} )\le u_{\bn}(\tau) \mid  \|\BX(\0) \|>
  u_{\bn}(\tau)) \,,
$$
which, in conjunction with \eqref{e:corners.tf.1}, proves both
existence of $\theta_{\rm run, \0}$ and the fact that it is equal to
$\theta_{\rm tf, \0}$. 
\end{proof}

Another version of a tail field based extremal index arises naturally
in limit theorems discussed in the next section. Let $\prec$ be an
invariant order on $\Z^k$. 

\begin{definition}  \label{def:tf.half}
Let $\BXRF$ be a stationary jointly regularly varying $\R^d$-valued
random field $\BXRF$ with the tail field $\BYRF$. Its {\bf half space
  extremal index} $\theta_{\rm half}$ is
$$
\theta_{\rm half} = \Prob\left( \sup_{\bt \prec\0} \|\BY(\bt)\|\leq 1\right) \,.
$$
\end{definition}

We will see in the next section that, under condition
\eqref{e:corners.tf.2}, the block extremal index exists and equals the 
half space   extremal index. A corollary of this is that the half
space   extremal index is independent of the invariant order $\prec$
as long as \eqref{e:corners.tf.2} holds for some array $(\br_{\bn})$. 

\begin{example} \label{ex:corners}
A simple class of models is that of max-moving averages with local
interaction. We consider one such model with two-dimensional time. Let 
$a_{-1,-1},a_{-1,1},a_{1,1},a_{1,-1}$ be numbers in $[0,1]$. 
Starting  with i.i.d. standard Fr\'echet(1) random variables 
$(Z(\bt) : \bt \in \Z^2)$, we define a stationary random field
$(X(\bt):\bt \in \Z^2)$ by 
$$
X(\bt) = 
\max\Bigl\{Z(\bt), a_{-1,-1}Z(\bt-\ones),a_{-1,1}Z(t_1-1,t_2+1),
a_{1,1}Z(\bt+\ones), a_{1,-1}Z(t_1+1,t_2-1)\Bigr\}\,.
$$
If $F_Z$ denotes the c.d.f. of a standard Fr\'echet(1) random
variable, then for any $u>0$, 
\begin{align*}
	&\Prob(M_X(\rec^+_{\br_{\bn}}) \le u) = (F_Z(u))^{E(\br,\ba)}\,,
\end{align*}
where
\begin{align*}
E(\br,\ba)&=r_{n_1}r_{n_2}+3(a_{-1,-1}+a_{-1,1}+a_{1,1}+a_{1,-1})
+(r_{n_1}-2)\bigl( \max(a_{-1,-1},a_{1,-1})+\max(a_{1,1},a_{-1,1})
\bigr) \\
& +(r_{n_2}-2) \bigl( \max(a_{-1,-1},a_{-1,1})+\max(a_{1,1},a_{1,-1})
\bigr)\,,
\end{align*}
while 
\begin{align*}
\Prob(X(\0) \le u) =   (F_Z(u))^{1+(a_{-1,-1}+a_{-1,1}+a_{1,1}+a_{1,-1})}.
\end{align*}
By \eqref{eq:u} and \eqref{e:u1} we conclude that the classical
extremal index exists, and
$$
\theta_{\rm cl} = ( 1+s )^{-1}\,,
$$
where
$$
s= a_{-1,-1}+a_{-1,1}+a_{1,1}+a_{1,-1}\,,
$$
and by \eqref{eq:theta}, the block extremal index $\theta_{\rm b}$
also exists and is equal to the classical extremal index. 

It is also easy to compute the run extremal index. We perform the
computation for the corner determined by $\bi=\0$, and it can be done
analogously for the other corners. Notice that  
\begin{align*}
&\lim_{\bn \to \binfty}\Prob(M_X(\rec^+_{\br_{\bn}}\backslash \{\0\})
  > u_{\bn}(\tau) \mid X(\0) > u_{\bn}(\tau))  \\
&= \lim_{\bn \to \binfty}\Prob(M_X(\{ \ones,(2,0),(0,2),(2,2))
  > u_{\bn}(\tau) \mid X(\0) > u_{\bn}(\tau))  \\
&= \lim_{\bn \to \binfty} \Bigl[ \Prob(Z(\0)>u_{\bn}(\tau) \mid X(\0) > u_{\bn}(\tau))
 \Prob(  X(\ones)  > u_{\bn}(\tau) \mid Z(\0) > u_{\bn}(\tau)\bigr) \\
&+ \Prob(a_{-1,1}Z((-1,1))>u_{\bn}(\tau) \mid X(\0) > u_{\bn}(\tau))
 \Prob(  X((0,2))  > u_{\bn}(\tau) \mid a_{-1,1}Z((-1,1)) >
  u_{\bn}(\tau)\bigr) \\
&+ \Prob(a_{1,1}Z(\ones)>u_{\bn}(\tau) \mid X(\0) > u_{\bn}(\tau))
 \Prob(   X(\ones)  > u_{\bn}(\tau) \mid a_{1,1}Z(\ones) >
  u_{\bn}(\tau)\bigr) \\
&+ \Prob(a_{1,-1}Z((1,-1))>u_{\bn}(\tau) \mid X(\0) > u_{\bn}(\tau))
 \Prob(  X((2,0))  > u_{\bn}(\tau) \mid a_{1,-1}Z((1,-1)) >
  u_{\bn}(\tau)\bigr) \Bigr]\\
&= \frac{1}{1+s}  \lim_{\bn \to \binfty} \Prob(  a_{-1,-1}Z(\0)  >
  u_{\bn}(\tau) \mid Z(\0) > u_{\bn}(\tau)\bigr) \\
&+ \frac{a_{-1,1}}{1+s}  \lim_{\bn \to \binfty} \Prob(  a_{-1,-1}Z((-1,1))  >
  u_{\bn}(\tau) \mid a_{-1,1} Z((-1,1)) > u_{\bn}(\tau)\bigr) 
+ \frac{a_{1,1}}{1+s}   \\
&+ \frac{a_{1,-1}}{1+s}  \lim_{\bn \to \binfty} \Prob(  a_{-1,-1}Z((1,-1))  >
  u_{\bn}(\tau) \mid a_{1,-1} Z((1,-1)) > u_{\bn}(\tau)\bigr) \\
&= (1+s)^{-1}\Bigl[ a_{-1,-1} + \min(a_{-1,1}, a_{-1,-1}) +  a_{1,1} + \min(a_{1,-1}, a_{-1,-1})
\Bigr]\,,
\end{align*}
 which equals, by definition, to $1-\theta_{\rm run,\0}$. 

Choosing $a_{-1,-1}=.1,a_{-1,1}=.7,a_{1,1}=.6,a_{1,-1}=.1$ results in
$\theta_{\rm run,\0}=.64,
\theta_{\rm run, \ones}=.44, \theta_{\rm run, (0,1)}=.4, \theta_{\rm
  run, (1,0)}=.6$, so the run
extremal index is different at all 4 corners. In this case
also $\theta_{\rm cl}=.4$. However, taking the equal weight mixture of
the above model with the model corresponding to
$a_{-1,-1}=.6,a_{-1,1}=.2,a_{1,1}=.6,a_{1,-1}=.1$ results in all 5
different indices: $\theta_\0=.52,
\theta_\ones=.42, \theta_{(0,1)}=.56, \theta_{(1,0)}=.7$ and
$\theta_{\rm cl}=.4$. 

Finally, because of the local interaction, condition
\eqref{e:corners.tf.2} holds in this case for any $(\br_{\bn})$
increasing to $\bm \infty$ such that $\br_\bn   /\bn \to \0$, and so
by Proposition \ref{pr:corner.tail}, the tail field extremal indices
coincide with the run extremal indices. 
\end{example}

\medskip
\section{Extremal index and limit theorems for point processes} \label{sec:pp}
Armed with the understanding of the spatial extremal indices
developed in the previous section, we now proceed to study the
extremal clusters.  

Let $\BXRF$ be an $\R^d$-valued stationary random field, jointly
regularly varying with index $\alpha > 0$, and let $\BYRF$,
$\BThetaRF$ be its associated tail field and spectral field, 
respectively. Let, once again, $(\br_{\bn})$ and $(u_{\bn}(\tau))$  be
arrays such that $\br_\bn  /\bn \to \0$, and \eqref{eq:u} holds for 
$\tau>0$. Consider the spatial point process (on $(\Rbar)^d$, from
which we remove the origin) defined by 
\begin{align}
C_\bn=\sum_{\bt \in
  \rec^+_{\br_{\bn}}}\delta_{u_{\bn}(\tau)^{-1}\BX(\bt)}. \label{eq:cluster}
\end{align}
We call it the cluster process,  and we 
are interested in the weak limit of the conditional law of
$C_{\bn}$, given that it does not vanish, i.e. given the event that
$M_X(\rec^+_{\br_{\bn}}) > u_{\bn}(\tau)$. We view the weak limit of the
cluster process as describing, asymptotically, a single extreme
cluster of the random field. Theorem \ref{thm:cluster} describes the
latter under the following assumption, which implies, at once, the
condition of Proposition \ref{pr:corner.tail} for every corner of the
hypercube. 
\begin{condition}\label{cond:vanish}
For any array
$(u_{\bn}(\tau))$ satisfying  
\eqref{eq:u} for some $\tau>0$ and some array $(\br_{\bn})$
increasing to $\bm \infty$ such that $\br_\bn   /\bn \to \0$,
	\begin{align}
	\lim_{M \to \binfty}\limsup_{\bn \to
          \binfty}\Prob\left(M_X(\rec_{\br_{\bn}} 
\backslash \rec_{M\ones}) > u_{\bn} (\tau)\mid
          \|\BX(\0)\|>u_{\bn}(\tau)\right) = 0. \label{eq:vanish} 
	\end{align}
\end{condition}

Let $\prec$ be an arbitrary invariant order  on $\Z^k$.  The argument
in the following theorem follows a logic similar to that in Theorem
4.3 of \cite{basrak:segers:2009}. 
\begin{theorem}\label{thm:cluster}
	Let $\BXRF$ be a jointly regularly varying with index $\alpha
        > 0$, $\R^d$-valued stationary random field, satisfying
        Condition \ref{cond:vanish}. Then
        $\Prob(\lim_{\|\bt\|_{\infty}  \to \infty}\|\BY(\bt)\| =
        0)=1$. Moreover, the block extremal index
        $\theta_{\rm b}$ exists, is positive, and 
\begin{equation} \label{e:comp.th}
\theta_{\rm b}=\theta_{\rm half} =\E\left[\max_{\bt \succeq
    \0}\|\BTheta(\bt)\|^\alpha-\max_{\bt \succ
    \0}\|\BTheta(\bt)\|^\alpha\right]. 
\end{equation} 
Furthermore,  the conditional law of $C_{\bn}$ converges weakly in the
space of Radon measures on $(\Rbar)^d\setminus\{\0\}$ to the conditional law
of the point process  
\begin{align}
C=\sum_{\bt \in \Z^k}\delta_{\BY(\bt)} \label{eq:dist_C} 
\end{align}  
given that $\max_{\bt \prec \0} \|\BY(\bt)\| \le 1$. 
The Laplace functional of C under this conditional law can be expressed as 
	\begin{align} 
	\Psi_{C}(f)=&\E\left[\exp \left\{-\sum_{\bt \in \Z^k} f(\BY(\bt))  \right\} \,\middle|\, \max_{\bt \prec \0} \|\BY(\bt)\| \le 1\right]\nonumber\\
	=&\theta_{\rm half}^{-1}\int_{0}^{\infty}\E\left[\exp \left\{ -\sum_ {\bt \succeq \0}f(y\BTheta(\bt))\right\}\1\left(y\max_{\bt \succeq \0}\|\BTheta(\bt)\|>1\right)\right.\nonumber\\
	&\left.-\exp \left\{ -\sum_{\bt\succ\0}f(y\BTheta(\bt))\right\}\1\left(y\max_{\bt \succ \0}\|\BTheta(\bt)\|>1\right)\right]\,d(-y^{-\alpha}) \label{eq:laplace_C}
	\end{align}
	for any nonnegative continuous $f$ on $(\Rbar)^d\setminus\{\0\}$ with
        a compact support. 
\end{theorem}
\begin{proof}
For any $v >0$, by Condition  \ref{cond:vanish} and the regular
variation of $\|\BX(\0)\|$, it holds that 
	\begin{align*}
	\lim_{M \to \infty}\limsup_{\bn \to
          \binfty}\Prob\left(M_X(\rec_{\br_{\bn}} 
\backslash \rec_{M\ones}) > u_{\bn}(\tau)v \mid \|\BX(\0)\|>u_{\bn}(\tau)\right) = 0.
	\end{align*}
	Therefore, for any $\epsilon >0$ and $v > 0$, there  exists
        $M>0$ such that for all $K>M$, 
	\begin{align*}
	\Prob\left(M_Y(\rec_{K\ones}\backslash \rec_{M\ones})>v \right) \le \epsilon.
	\end{align*}
	This implies that $\Prob(\lim_{\|\bt\|_{\infty}  \to
          \infty}\|\BY(\bt)\| = 0)=1$. Next, choose an integer $M$ so large that 
$$
\limsup_{\bn \to
          \binfty}\Prob\left(M_X(\rec_{\br_{\bn}} 
\backslash \rec_{M\ones}) > u_{\bn} (\tau)\mid
          \|\BX(\0)\|>u_{\bn}(\tau)\right) \leq 1/2\,.
$$	
	Let $\gamma_\ell=  \lfloor (r_{\bn})_\ell/M \rfloor, \,
        \ell=1,\ldots, k$, and fit into the hypercube
        $\rec^+_{\br_{\bn}}$ the $\prod_{\ell = 1}^k \gamma_\ell$
        smaller hypercubes 
with $M$ points on each side. We decompose the event that a value 
exceeding $u_{\bn}(\tau)$ is attained at one of the points of the
resulting grid according to the last
point of the grid (in the lexicographic order) at which a value
exceeding $u_{\bn}(\tau)$ is attained. For a point $M\bp$ on this
grid, let $A_{M\bp}$ denote the set of the points of the grid larger than
$M\bp$. By stationarity,
	\begin{align*}
	&\Prob(M_X(\rec^+_{\br_{\bn}})>u_{\bn}(\tau)) \\
	\ge &\sum_{p_1=0}^{\gamma_1-1}\cdots\sum_{p_k=0}^{\gamma_k-1} 
\Prob\left(\|\BX(M\bp)\|>u_{\bn}(\tau),M_X(A_{M\bp})\le u_{\bn}(\tau)\right)\\
	= &\sum_{p_1=0}^{\gamma_1-1}\cdots\sum_{p_k=0}^{\gamma_k-1}
            \left[  \Prob\left(\|\BX(M\bp)\|>u_{\bn}(\tau)\right)-
            \Prob\left(\|\BX(M\bp)\|>u_{\bn}(\tau),M_X(A_{M\bp})>
            u_{\bn}(\tau)\right)\right]\\ 
	\ge & \left( \prod_{\ell = 1}^k \gamma_{\ell} \right) \left[\Prob(\|\BX(\0)\|>u_{\bn}(\tau)) -\Prob\left(\|\BX(\0)\|>u_{\bn}(\tau),M_X(\rec_{\br_{\bn}} \backslash \rec_{M\ones})> u_{\bn}(\tau)\right)\right]\,, 
	\end{align*}
 so
\begin{align} \label{e:lower.b.th}
 \liminf_{n\to\infty}  \frac{\Prob(M_X(\rec^+_{\br_{\bn}})>u_{\bn}(\tau))}{\left(\prod_{\ell
          = 1}^k r_{n_\ell}\right)\Prob(\|\BX(\0)\|>u_{\bn}(\tau))}  
\geq 2^{-1}M^{-k}>0,.
\end{align}

Next, we decompose the event $M_X(\rec^+_{\br_{\bn}}) > u_{\bn}(\tau)$
using the  order $\prec$,  we have 
	\begin{align*}
	&\E\left[ \exp \left\{ -\sum_{\bi \in \rec^+_{\br_{\bn}}} f(u_{\bn}(\tau)^{-1}\BX(\bi))\right\}\1(M_X(\rec^+_{\br_{\bn}})>u_{\bn}(\tau))	\right]\\
	=& \sum_{\bt \in \rec^+_{\br_{\bn}}}\E\left[ \exp \left\{ -\sum_{\bi \in \rec^+_{\br_{\bn}}} f(u_{\bn}(\tau)^{-1}\BX(\bi))\right\}\1\left(\sup_{\bs \prec \bt,\bs \in \rec^+_{\br_{\bn}}} \|\BX(\bs)\| \le u_{\bn}(\tau)<\|\BX(\bt)\|\right)\right],
	\end{align*}
with the convention that the supremum over the empty set is defined to
be equal to zero. Denote
$$
\theta_{\bn}=
\frac{\Prob(M_X(\rec^+_{\br_{\bn}})>u_{\bn}(\tau))}{\left(\prod_{\ell
          = 1}^k r_{n_\ell}\right)\Prob(\|\BX(\0)\|>u_{\bn}(\tau))}
    \,,
$$
so that, if the array $(\theta_{\bn})$ has a limit as $\bn\to\binfty$,
the limit is the block extremal index. It follows from
\eqref{e:lower.b.th} that every subsequential limit of this array is
strictly positive. 
Fix  $\bbm \in \N^k$, and choose $\bn$ large
enough so that $\br_{\bn} \ge 2\bbm-\ones$. Then 
	\begin{align}
	&\left|\E\left[ \exp \left\{ -\sum_{\bi \in \rec^+_{\br_{\bn}}} f(u_{\bn}(\tau)^{-1}\BX(\bi))\right\} \,\middle|\, M_X(\rec^+_{\br_{\bn}})>u_{\bn}(\tau)	\right]\right.\nonumber\\
	& -\theta_{\bn}^{-1}\left.\E\left[ \exp \left\{ -\sum_{\bi \in
          \rec_{\bbm}} f(u_{\bn}(\tau)^{-1}\BX(\bi))\right\}
          \1\left(\max_{\bt \prec \0, \bt \in \rec_{\bbm}}
          \|\BX(\bt)\| \le u_{\bn}(\tau)\right) \,\middle|\,
          \|\BX(\0)\|>u_{\bn}(\tau)\right]\right| \nonumber \\
	\le &\frac{1}{\Prob(M_X(\rec^+_{\br_{\bn}})>u_{\bn}(\tau))}
             \label{eq:diff}  \\
             & \sum_{\bt \in \rec^+_{\br_{\bn}}}\left|\E\left[ \exp
              \left\{ -\sum_{\bi \in \rec^+_{\br_{\bn}}}
              f(u_{\bn}(\tau)^{-1}\BX(\bi))\right\}\1\left(\max_{\bs \preceq
              \bt, \bs \in \rec^{+}_{\br_{\bn}}} \|\BX(\bs)\| \le
              u_{\bn}(\tau)<
              \|\BX(\bt)\|\right)\right]\right. \nonumber 
          \\
	&-\left.\E\left[ \exp \left\{ -\sum_{\bi \in \rec_{\bbm}} f(u_{\bn}(\tau)^{-1}\BX(\bi))\right\} \1\left( \max_{\bs \prec \0, \bs \in \rec_{\bbm}} \|\BX(\bs)\| \le u_{\bn}(\tau) < \|\BX(\0)\|\right)\right]\right|\,. \nonumber
	\end{align}
Let  $\calI =[\bbm-\ones : \br_{\bn}-\bbm]$. By stationarity and
invariance of the order, 
	\begin{align*}
	&\sum_{\bt \in \calI}\left|\E\left[ \exp \left\{ -\sum_{\bi
          \in \rec^+_{\br_{\bn}}}
          f(u_{\bn}(\tau)^{-1}\BX(\bi))\right\}\1\left(\max_{\bs \preceq
          \bt, \bs \in \rec^{+}_{\br_{\bn}}} \|\BX(\bs)\| \le u_{\bn}(\tau)<
          \|\BX(\bt)\|\right)\right]\right.  \\
	&-\left.\E\left[ \exp \left\{ -\sum_{\bi \in \rec_{\bbm}} f(u_{\bn}(\tau)^{-1}\BX(\bi))\right\} \1\left( \max_{\bs \prec \0, \bs \in \rec_{\bbm}} \|\BX(\bs)\| \le u_{\bn}(\tau) < \|\BX(\0)\|\right)\right]\right| \nonumber\\
	=&\sum_{\bt \in \calI}\left|\E\left( \left[ \exp \left\{ -\sum_{\bi
           \in \rec^+_{\br_{\bn}}}
           f(u_{\bn}(\tau)^{-1}\BX(\bi))\right\}\1\left(\max_{\bs \preceq
           \bt, \bs \in \rec^{+}_{\br_{\bn}}} \|\BX(\bs)\| \le
           u_{\bn}(\tau)<\|\BX(\bt)\|\right)\right]\right.\right. \\
	&-\left.\left.\left[ \exp \left\{ -\sum_{\bi \in \rec_{\bbm}(\bt)}
          f(u_{\bn}(\tau)^{-1}\BX(\bi))\right\} \1\left( \max_{\bs \prec
          \bt, \bs \in \rec_{\bbm}(\bt)} \|\BX(\bs)\| \le u_{\bn}(\tau) <
          \|\BX(\bt)\|\right)\right]\right|\right)\,. 
	\end{align*}
Since $f$ vanishes in a neighbourhood of the origin, there is $0<v \le
1$ such that $f(\bx) = 0$ when $\|\bx\| \le v$. Therefore, for each fixed $\bt \in
\calI$, the difference in the sum above 
will be nonzero only if
$u_{\bn}(\tau)^{-1}\|\BX(\bs)\|> v$ for some $\bs \in
(\rec^+_{\br_{\bn}}\backslash \rec_{\bbm}(\bt))$. By stationarity, this sum is upper
bounded by 
	\begin{align*}
	  & \sum_{\bt \in \calI}\Prob\left(M_X(\rec^+_{\br_{\bn}}\backslash\rec_{\bbm}(\bt))>u_{\bn}(\tau)v, \|\BX(\bt)\|>u_{\bn}(\tau)\right)\nonumber\\
	\le &\left(\prod_{\ell=1}^k r_{n_\ell} \right) \Prob\left(M_X(\rec_{\br_{\bn}}\backslash \rec_{\bbm})>u_{\bn}(\tau)v, \|\BX(\0)\|>u_{\bn}(\tau)\right).
	\end{align*}
	On the other hand, for each $\bt \in \rec^+_{\br_{\bn}} \backslash \calI$, the summand is upper bounded by
	\begin{align*}
	\E\left[ \exp \left\{ -\sum_{\bs \in \rec_{\bbm}} f(u_{\bn}(\tau)^{-1}\BX(\bs))\right\} \1\left( \max_{\bs \prec \0, \bs \in \rec_{\bbm}} \|\BX(\bs)\| \le u_{\bn}(\tau) < \|\BX(\0)\|\right)\right] \le \Prob(\|\BX(\0)\|>u_{\bn}(\tau)).
	\end{align*}
Combining the two parts, we see that the difference in \eqref{eq:diff}
does not exceed  
	\begin{align*}
	&\frac{1}{\Prob(M_X(\rec^+_{\br_{\bn}})>u_{\bn}(\tau))} \left[
          \left(\prod_{\ell=1}^k r_{n_\ell} \right)
          \Prob\left(M_X(\rec_{\br_{\bn}}\backslash
          \rec_{\bbm})>u_{\bn}(\tau)v,
          \|\BX(\0)\|>u_{\bn}(\tau)\right)\right. \\
        &\left. +{\rm Card}(\rec^+_{\br_{\bn}} \backslash
          \calI)\Prob(\|\BX(\0)\|>u_{\bn}(\tau))\right]\nonumber\\ 
	&=\frac{1}{\theta_{\bn}}\left[\Prob\left(M_X(\rec_{\br_{\bn}}\backslash \rec_{\bbm})>u_{\bn}(\tau)v \mid \|\BX(\0)\|>u_{\bn}(\tau)\right)+\frac{{\rm Card}(\rec^+_{\br_{\bn}} \backslash \calI)}{\prod_{\ell=1}^k r_{n_l}}\right] \to 0
	\end{align*}
	as $\bn \to \binfty, \bbm \to \binfty$, where we have used
        \eqref{e:lower.b.th}. Therefore, for any sequence $(\bn_k)$
        converging to $\binfty$, 
        along which $\theta_\bn$ has a (positive) limit, say, $L$, 
	\begin{align*}
	&\lim_{k \to \infty}\E\left[ \exp \left\{ -\sum_{\bt\in
          \rec^+_{\br_{\bn_k}}}
          f(u_{\bn_k}(\tau)^{-1}\BX(\bt))\right\}\,
\middle|\, M_X(\rec^+_{\br_{\bn_k}})>u_{\bn_k}(\tau)	\right]\\
	=&\lim_{\bbm \to \binfty}\lim_{k \to  \infty}\theta_{\bn_k}^{-1}
\E \Biggl[
 \exp \left\{ -\sum_{\bt \in \rec_{\bbm}}
           f(u_{\bn_k}(\tau)^{-1}\BX(\bt))\right\} \\
&\hskip 1in\1\left( \max_{\bt
           \prec \0, \bt \in \rec_{\bbm}}\|\BX(\bt)\|\le
           u_{\bn_k}(\tau) \right)\,\Bigg|\,
           \|\BX(\0)\|>u_{\bn_k}(\tau)\Biggr]
\\
	=&L^{-1}\E\left[ \exp \left\{ -\sum_{\bt \in \Z^k} f(\BY(\bt))\right\} \1\left( \max_{\bt \prec \0}\|\BY(\bt)\|\le 1 \right)\right].
	\end{align*}
	Choosing $f = 0$ gives us 
	\begin{align*}
	L = \Prob\left(\max_{\bt \prec \0}\|\BY(\bt)\| \le
          1\right), 
	\end{align*}
which implies several things. First of all, it implies that all
subsequential limits $L$ are equal, so the array $(\theta_{\bn})$ has
a limit  as $\bn\to\binfty$. Therefore the block extremal index
exists and is positive, and $\theta_{\rm   b}=\theta_{\rm half}$. This
also proves the convergence of the Laplace transform of the cluster
process computed under its conditional law: 
\begin{align*}
	\lim_{\bn \to \binfty} &\E\left[ \exp \left\{ -\sum_{\bi \in
  \rec^+_{\br_{\bn}}} f(u_{\bn}(\tau)^{-1}\BX(\bi))\right\}
  \,\middle|\, M_X(\rec^+_{\br_{\bn}})>u_{\bn}(\tau)	\right]  \\ 
= &\E\left[ \exp \left\{ -\sum_{\bt \in \Z^k} f(\BY(\bt))\right\}
  \,\middle|\, \max_{\bt \prec \0}\|\BY(\bt)\| \le 1 \right] 
\end{align*}
for any nonnegative continuous $f$ on $(\Rbar)^d\setminus\{\0\}$ with
        a compact support. This, of course, proves the stated weak
        convergence of  the conditional laws of the cluster process.  
	
One shows that the Laplace transform $\Psi_C(f)$ of
 the limiting point process computed under its conditional law 
has the expression in the right hand side of
\eqref{eq:laplace_C} using the same argument as in  
\cite{basrak:segers:2009},  using the invariant order
$\prec$.  Finally, \eqref{e:comp.th} follows from \eqref{eq:laplace_C}
applied to the zero function. 
\end{proof}

\begin{remark} \label{rk:more.stuff}
It is elementary to check that, if $f(\bx)=0$ whenever $\|\bx\|\leq
1$, then the obvious analogue of an alternative expression (4.6) in
\cite{basrak:segers:2009} for the Laplace transform $\Psi_C(f)$ holds
as well. Furthermore, under both Condition \ref{cond:vanish} and the
asymptotic independence of extremal clusters condition 
	\begin{align*}
	\E\left[\exp \left\{-\sum_{\bt\in \rec^+_{\bn}}
          f(u_{\bn}(1)^{-1}\BX(\bt))\right\}\right]-\left(\E\left[\exp
          \left\{-\sum_{\bt \in \rec^+_{\br_{\bn}}}
          f(u_{\bn}(1)^{-1}\BX(\bt))
          \right\}\right]\right)^{\prod_{\ell = 1}^k \lfloor {n_\ell}/r_{n_\ell}\rfloor} \to 0
	\end{align*} 
for every continuous function $f$ with a compact support, 
one also obtains a picture of exceedance clusters on a larger
scale, as in Theorem 4.5 {\it ibid.} For the point process 
\begin{align*}
N_n=\sum_{\bt\in \rec^+_{\bn}}\delta_{u_{\bn}(1)^{-1}\BX(\bt)}
\end{align*}
one obtains weak convergence in the
space of Radon measures on $(\Rbar)^d\setminus\{\0\}$ to a cluster Poisson
point process whose restriction to the set $\{ \bx: \, \|\bx\|>a\}$,
$a>0$, has the representation
$$
 \sum_{i=1}^{P_a} \sum_{\bt\in\Z^k}
\delta_{a\BZ_{i}(\bt)}
\one\bigl( \|\BZ_{i}(\bt)\|>1\bigr)\,,
$$
where $\bigl( \BZ_{i}(\bt), \,  \bt\in\Z^k\bigr), \, i=1,2,\ldots$ are
i.i.d. copies of the single cluster limiting process in Theorem
\ref{thm:cluster}, independent of a mean $\theta_{\rm b}u^{-\alpha}$
Poisson random variable  $P_a$. A different, and very detailed, 
representation of the entire limiting point process is in
\cite{basrak:plannic:2018}. 
\end{remark}

\section{Brown-Resnick Random Fields} \label{sec:brf}
The tail field is a convenient formalism to describe the extremes of a
jointly regularly varying stationary random field. It is useful, in
particular, in describing the extremal clusters, and it can be used to
define versions of the extremal index. In order to make it concrete,
in this section, we focus on the class of the so-called Brown-Resnick
random fields. For simplicity we will keep the values of the field
one-dimensional, with the standard Fr\'echet marginal distributions. 

Let $(W(\bt): \bt \in \R^k)$ be a stationary increment (real-valued)
zero-mean Gaussian 
random field, with variance $\sigma^2(\bt)$ and variogram
$\gamma(\bt) =E(W(\bt)-W(\0))^2$, $\bt\in\R^k$. The stationarity of
the increments means that $E(W(\bt)-W(\bs))^2=\gamma(\bt-\bs)$ for all
$\bt,\bs\in\R^k$. Let $(W_i(\bt): \bt \in \R^k), i \in \N$ be
i.i.d. copies of this random field, independent of a  Poisson point process
$\sum_{i=1}^\infty \delta_{U_i}$ on $\R_+$ with intensity
$du/u^2$. The Brown-Resnick random field associated with the Gaussian
random field $(W(\bt): \bt \in \R^k)$ is defined by 
\begin{align}
	X(\bt) = \max_{i=1,2,\dots} U_i\exp\{W_{i}(\bt) -
  \sigma^2(\bt)/2\}\,. \label{eq:BRRF} 
	\end{align}
Since $E\bigl(\exp\{W(\bt) -  \sigma^2(\bt)/2\} \bigr)  =1$ for each $t$,
this is a well defined max-stable random field with the standard  Fr\'echet
marginal distributions; see \cite{dehaan:1984}. Furthermore, it is a
stationary random field (even when the Gaussian
random field $(W(\bt): \bt \in \R^k)$ itself is not stationary); see Theorem
2 and Remark 3 in \cite{kabluchko:schlather:dehaan:2009}. As any
max-stable random field with the standard Fr\'echet  marginal
distributions, the Brown-Resnick random field is multivariate regular
varying (with $\alpha=1$). This fact is also seen from the following
proposition, that computes the law of the tail field of this random
field. 

\begin{proposition} \label{thm:BR_Y}
	The Brown-Resnick random field $\XRF$ is multivariate
        regularly varying, and the finite-dimensional distributions of
        its tail field $\YRF$ can be computed by  
	\begin{align}
	&\Prob(Y(\bt_1)\leq y_1,\dots,Y(\bt_n)\leq y_n)\nonumber \\
	=&\E\left[ \max_{i=1,\dots,n} \left(\frac{1}{y_i}\exp\left\{{W(\bt_i)-\frac{\sigma^2(\bt_i)}{2}}\right\}, \exp\left\{{W(\0)-\frac{\sigma^2(\0)}{2}} \right\} \right)\right]\nonumber \\
	&- \E\left[ \max_{i=1,\dots,n} \frac{1}{y_i}\exp\left\{{W(\bt_i)-\frac{\sigma^2(\bt_i)}{2}}\right\} \right]	\label{eq:BR_Y_dist}
	\end{align}
for $\bt_1,\dots,\bt_n \in \Z^k$ and positive $y_1,\ldots, y_n$. In
particular, the marginal distributions of the tail field are given by 
	\begin{align}
	\Prob(Y(\bt)\leq y) = \Phi\left( \frac{2\ln y +
          \gamma(\bt)}{2\sqrt{\gamma(\bt)}} \right) -
          \frac{1}{y}\Phi\left( \frac{2\ln
          y-\gamma(\bt)}{2\sqrt{\gamma(\bt)}} \right)
          ,	\label{eq:BR_Yt} 
	\end{align}
for $\bt\in \Z^k$ and $y>0$. Here  $\Phi(\cdot)$ is the standard normal cdf.
\end{proposition}

\begin{proof}
	Let $V_i(\bt) = \exp\{W_i(\bt) - \sigma^2(\bt)/2\}$. Then for
        any finite set of points in $\Z^k$ and positive numbers, 
\begin{align}
	\Prob(X(\bt_1) \le x_1,\dots,X(\bt_n) \le x_n) &= \exp\left\{
       -\E\left[ \max \left(\frac{V(\bt_1)}{x_1},\dots,\frac{V(\bt_n)}{x_n}\right)
        \right] \right\}, 
\label{eq:X_dist}
\end{align}
so 
	\begin{align*}
		&\Prob(x^{-1}X(\bt_1) \le y_1,\dots,x^{-1}X(\bt_n) \le y_n \mid X(\0)>x)\\
		=&\frac{\Prob(X(\bt_1) \le xy_1,\dots,X(\bt_n) \le
                   xy_n) - \Prob(X(\bt_1) \le xy_1,\dots,X(\bt_n) \le
                   xy_n, X(\0)\leq x)}{P(X(\0)>x)}\\
		 &=\frac{\exp\left\{ -\E\left[ \max \left(\frac{V(\bt_1)}{xy_1},\dots,\frac{V(\bt_n)}{xy_n}\right)  \right] \right\}- \exp\left\{ -\E\left[ \max \left(\frac{V(\bt_1)}{xy_1},\dots,\frac{V(\bt_n)}{xy_n},\frac{V(\0)}{x}\right)  \right] \right\}}{1-e^{-1/x}}\\
		 \sim &x\left[\exp\left\{ -\frac{1}{x}\E\left[ \max_{i=1,\dots,n}\; \frac{1}{y_i}\exp\left\{{W(\bt_i)-\frac{\sigma^2(\bt_i)}{2}}\right\} \right] \right\}\right.\\
		 &\left.- \exp\left\{ -\frac{1}{x}\E\left[
                   \max_{i=1,\dots,n}
                   \left(\frac{1}{y_i}\exp\left\{{W(\bt_i)-\frac{\sigma^2(\bt_i)}{2}}\right\},
                   \exp\left\{{W(\0)-\frac{\sigma^2(\0)}{2}} \right\}
                   \right)\right] \right\}\right], 
	\end{align*}
which converges, as $x\to\infty$, to the expression in the right hand
side of \eqref{eq:BR_Y_dist}. In particular, the marginal
distributions satisfy 
	\begin{align*}
	 &\Prob(Y(\bt)\le y) \\
	 =&\E\left[ \max
            \left(\frac{1}{y}\exp\left\{{W(\bt)-\frac{\sigma^2(\bt)}{2}}\right\},
            \exp\left\{{W(\0)-\frac{\sigma^2(\0)}{2}} \right\}
            \right)\right] -
            \E\left[\frac{1}{y}\exp\left\{{W(\bt)-\frac{\sigma^2(\bt)}{2}}\right\}\right], 
	\end{align*}
and \eqref{eq:BR_Yt} follows by straightforward calculations with
lognormal random variables; see e.g. \cite{lien:1986}. 
\end{proof}

We will investigate the extremal behaviour of the restriction of
the Brown-Resnick random field to the integer grid $\Z^k$. The first
question is whether this field  satisfies  Condition
\ref{cond:vanish} (and, hence, also the
assumption \eqref{e:corners.tf.2}). The answer is given in the
following proposition. 
\begin{proposition} \label{pr:BR.local}
Let $\XRF$ be the Brown-Resnick random field \eqref{eq:BRRF} 
corresponding to a
stationary increment zero-mean Gaussian 
random field with variance $\sigma^2(\bt)$. Then $\XRF$ satisfied 
 Condition \ref{cond:vanish} if and only if the Gaussian field
 satisfies
\begin{equation} \label{e:Gauss.vanish}
\lim_{\bt\to\binfty, \, \bt\in \Z^k} \bigl( W_{i}(\bt) -
  \sigma^2(\bt)/2\bigr) = -\infty \ \ \text{a.s..}
\end{equation}
\end{proposition} 
\begin{proof}
Choose and fix the arrays $(u_{\bn}(\tau))$ and $(\br_{\bn})$. By the
inclusion-exclusion formula,
\begin{align*}
&\Prob\left(M_X(\rec_{\br_{\bn}} 
\backslash \rec_{M\ones}) > u_{\bn} (\tau)\mid
          \|\BX(\0)\|>u_{\bn}(\tau)\right) \\
&= 1- \frac{\Prob\left(M_X(\rec_{\br_{\bn}} 
\backslash \rec_{M\ones}\cup \{\0\}) > u_{\bn} (\tau)\right)
- \Prob\left(M_X(\rec_{\br_{\bn}} 
\backslash \rec_{M\ones}) > u_{\bn} (\tau)\right)}{\Prob\left(
  \|\BX(\0)\|>u_{\bn}(\tau)\right)}\,,
\end{align*}
so Condition \ref{cond:vanish} is satisfied if and only if 
\begin{align} \label{e:lim1}
\lim_{M \to \binfty}\liminf_{\bn \to \binfty}u_{\bn}(\tau)\left[
\Prob\left(M_X(\rec_{\br_{\bn}} 
\backslash \rec_{M\ones}\cup \{\0\}) > u_{\bn} (\tau)\right)
- \Prob\left(M_X(\rec_{\br_{\bn}} 
\backslash \rec_{M\ones}) > u_{\bn} (\tau)\right)\right]=1\,.
\end{align}
By \eqref{eq:X_dist}, as $\bn\to\binfty$, 
\begin{align*}
&\Prob\left(M_X(\rec_{\br_{\bn}} 
\backslash \rec_{M\ones}\cup \{\0\}) > u_{\bn} (\tau)\right)
- \Prob\left(M_X(\rec_{\br_{\bn}} 
\backslash \rec_{M\ones}) > u_{\bn} (\tau)\right) \\
=&\exp\left\{ -u_{\bn} (\tau)^{-1} \E\max_{\bt\in \rec_{\br_{\bn}} 
\backslash \rec_{M\ones}}V(\bt)\right\}-
\exp\left\{ -u_{\bn} (\tau)^{-1} \E\max_{\bt\in \rec_{\br_{\bn}} 
\backslash \rec_{M\ones}\cup\{\0\}}V(\bt)\right\}\\
\sim & u_{\bn} (\tau)^{-1} \left[ \E\max_{\bt\in \rec_{\br_{\bn}} 
\backslash \rec_{M\ones}\cup\{\0\}}V(\bt)-
\E\max_{\bt\in \rec_{\br_{\bn}} 
\backslash \rec_{M\ones}}V(\bt)\right]
\exp\left\{ -u_{\bn} (\tau)^{-1} \E\max_{\bt\in \rec_{\br_{\bn}} 
\backslash \rec_{M\ones}}V(\bt)\right\}\,.
\end{align*}
Since 
\begin{align*}
0& \leq u_{\bn} (\tau)^{-1} \E\max_{\bt\in \rec_{\br_{\bn}} 
\backslash \rec_{M\ones}}V(\bt) 
\leq u_{\bn} (\tau)^{-1} \sum_{\bt\in \rec_{\br_{\bn}} 
\backslash \rec_{M\ones}}\E V(\bt) \\
& = u_{\bn} (\tau)^{-1} {\rm Card}\bigl( \rec_{\br_{\bn}} 
\backslash \rec_{M\ones}\bigr)\to 0\,,
\end{align*}
\eqref{e:lim1} is equivalent to 
\begin{align} \label{e:lim1a}
\lim_{M \to \binfty}\liminf_{\bn \to \binfty} \, \E\bigl[ \max_{\bt\in \rec_{\br_{\bn}} 
\backslash \rec_{M\ones}\cup\{\0\}}V(\bt)-
\max_{\bt\in \rec_{\br_{\bn}} 
\backslash \rec_{M\ones}}V(\bt)\bigr]=1\,.
\end{align}

Suppose first that \eqref{e:Gauss.vanish} holds, i.e. that
$V(\bt)\to 0$ a.s. as $\bt\to\binfty$. Then 
\begin{align*} 
&\lim_{M \to \binfty}\liminf_{\bn \to \binfty} \, \E\bigl[ \max_{\bt\in \rec_{\br_{\bn}} 
\backslash \rec_{M\ones}\cup\{\0\}}V(\bt)-
\max_{\bt\in \rec_{\br_{\bn}} 
\backslash \rec_{M\ones}}V(\bt)\bigr]\\
= &\lim_{M \to \binfty} \E\bigl[ \max_{\bt\in (\rec_{M\ones})^c\cup\{\0\}}V(\bt)-
\max_{\bt\in (\rec_{M\ones})^c}V(\bt)\bigr] = \E V(\0)=1\,,
\end{align*}
so \eqref{e:lim1a} is satisfied. 

Suppose, on the other hand, that \eqref{e:Gauss.vanish} fails. Then
there is $a>0$ and an event $A$ of positive probability such that 
\begin{equation} \label{e:Aa}
\limsup_{\bt\to\binfty, \, \bt\in \Z^k} V(\bt)>a \ \ \text{on $A$.}
\end{equation}
Therefore, on $A$, for all $M$, 
\begin{align} \label{e:less.V}
\limsup_{\bn \to \binfty} 
\bigl[ \max_{\bt\in \rec_{\br_{\bn}} 
\backslash \rec_{M\ones}\cup\{\0\}}V(\bt)-
\max_{\bt\in \rec_{\br_{\bn}} 
\backslash \rec_{M\ones}}V(\bt)\bigr]\leq \max(V(\0),a)-a\,.
\end{align}
Since for every $\bn$ 
$$
\max_{\bt\in \rec_{\br_{\bn}} 
\backslash \rec_{M\ones}\cup\{\0\}}V(\bt)-
\max_{\bt\in \rec_{\br_{\bn}} 
\backslash \rec_{M\ones}}V(\bt)\leq V(\0)\,,
$$
an integrable random variable, we can use Fatou's lemma in the form
\begin{align*}
&\lim_{M \to \binfty}\liminf_{\bn \to \binfty} \, \E\bigl[ \max_{\bt\in \rec_{\br_{\bn}} 
\backslash \rec_{M\ones}\cup\{\0\}}V(\bt)-
\max_{\bt\in \rec_{\br_{\bn}} 
\backslash \rec_{M\ones}}V(\bt)\bigr]\\
\leq &\lim_{M \to \binfty}\limsup_{\bn \to \binfty} \, 
\E\bigl[ \max_{\bt\in \rec_{\br_{\bn}} 
\backslash \rec_{M\ones}\cup\{\0\}}V(\bt)-
\max_{\bt\in \rec_{\br_{\bn}} 
\backslash \rec_{M\ones}}V(\bt)\bigr]\\
\leq &\lim_{M \to \binfty}
\E \limsup_{\bn \to \binfty} \, \bigl[ \max_{\bt\in \rec_{\br_{\bn}} 
\backslash \rec_{M\ones}\cup\{\0\}}V(\bt)-
\max_{\bt\in \rec_{\br_{\bn}} 
\backslash \rec_{M\ones}}V(\bt)\bigr] \,.
\end{align*}
The upper limit inside the expectation cannot exceed $V(\0)$ and, by 
\eqref{e:less.V}, it is strictly smaller than $V(\0)$ on an event of a
positive probability. Therefore, 
$$
\lim_{M \to \binfty}\liminf_{\bn \to \binfty} \, \E\bigl[ \max_{\bt\in \rec_{\br_{\bn}} 
\backslash \rec_{M\ones}\cup\{\0\}}V(\bt)-
\max_{\bt\in \rec_{\br_{\bn}} 
\backslash \rec_{M\ones}}V(\bt)\bigr]<\E V(\0)=1\,,
$$
and \eqref{e:lim1a} fails. 
\end{proof}

Since the condition \eqref{e:Gauss.vanish} cannot be satisfied if the
Gaussian random field $(W(\bt): \bt \in \R^k)$ is stationary (and
nontrivial), Condition \ref{cond:vanish} is not satisfied for the
corresponding Brown-Resnick random field. Furthermore, denoting the
constant variance of the Gaussian field by $\sigma^2>0$, we have by \eqref
{eq:BR_Yt},
	\begin{align*}
	\Prob(Y(\bt) > 1) =  2\Phi\left( -\frac12\sqrt{\gamma(\bt)}
          \right) \ge 2\Phi(-\sigma). 
	\end{align*}
Therefore, the tail field does not necessarily vanish as $\bt \to
\binfty$, and the extremal clusters may last indefinitely. 

The following corollary is an immediate consequence of propositions
\ref{thm:BR_Y} and \ref{pr:BR.local}. 

\begin{corollary}\label{thm:BRRF_ei}
Let $\XRF$ be the Brown-Resnick random field \eqref{eq:BRRF} 
corresponding to a
stationary increment zero-mean Gaussian 
random field with variance $\sigma^2(\bt)$, satisfying 
\eqref{e:Gauss.vanish}. Then  the block extremal index
$\theta_{\rm b}$ exists, is positive and equal to $\theta_{\rm half}$ for
every invariant order  on $\Z^k$,  and can be computed by 
\begin{align}
	\theta_{\rm b} = \E\left[ \max_{\bt \preceq \0} \exp\left\{W(\bt) - \sigma^2(\bt)/2\right\}\right] - \E\left[\max_{\bt \prec \0} \exp\left\{W(\bt) - \sigma^2(\bt)/2\right\}  \right]. \label{eq:BRRF_theta}
	\end{align}
\end{corollary}
\begin{proof}
By Proposition \ref{pr:BR.local},  Condition \ref{cond:vanish} is
satisfied. By Theorem \ref{thm:cluster}  the block extremal index
$\theta_{\rm b}$ exists, is positive and, using \eqref{eq:BR_Y_dist}, 
	\begin{align*}
	\theta_{\rm b} &= \theta_{\rm half} = \Prob \left(\max_{\bt
                         \prec \0} Y(\bt) \le 1\right)\\ 
	& = \E\left[ \max_{\bt \preceq \0} \exp\left\{W(\bt) - \sigma^2(\bt)/2\right\}\right] - \E\left[\max_{\bt \prec \0} \exp\left\{W(\bt) - \sigma^2(\bt)/2\right\}  \right].
	\end{align*}
\end{proof}

Since an exact simulation of Brown-Resnick random fields is not easy
(see
e.g. \cite{dieker:mikosch:2015,oesting:kabluchko:schlather:2012}),
results of the type \eqref{eq:BRRF_theta} can be used for numerical
evaluation of the extremal index of the field.  We demonstrate this on
an example. 

\begin{example}[Brown-Resnick field corresponding to the additive
  Fractional Brownian motion] \label{ex:fbs}   
Recall that the standard Fractional Brownian motion with Hurst
paremeter $0<H<1$ is a stationary increment zero-mean Gaussian
process on $\R$, vanishing at the origin, with the variogram
$\gamma(t) = |t|^{2H}, \, t\in \R$. 
	Let $\fBm_{H_i}(t) \in \R$, $i=1,\ldots, k$ be independent
        standard 
        Fractional Brownian motions, with respective  Hurst parameters $H_1,
        \ldots, H_k$. Then  
	\begin{align*}
	W(t_1,\ldots, t_k) = \fBm_{H_1}(t_1) +\cdots +
          \fBm_{H_k}(t_k), \, \bt=(t_1,\ldots, t_k)\in \R^k
	\end{align*} 
is a zero mean stationary increment Gaussian random field, the additive
  Fractional Brownian motion.  It is elementary to check that each
  standard Fractional Brownian motion satisfies \eqref{e:Gauss.vanish}
  (this follows, for example, by the Borel-Cantelli lemma). Therefore,
  so does the additive   Fractional Brownian motion. 

For $k=2$ we have used \eqref{eq:BRRF_theta} to calculate the block
extremal index of the Brown-Resnick random field corresponding to the
additive  Fractional Brownian motion. In this calculation we
truncated the domain of the additive  Fractional Brownian motion to
the square $[-200,200]\times [-200,200]$. The results are plotted on
Figure \ref{fig:fbmsim}  as a function of $H_1$ and $H_2$. 

	\begin{figure}[h]
		\includegraphics[width=0.7\textwidth]{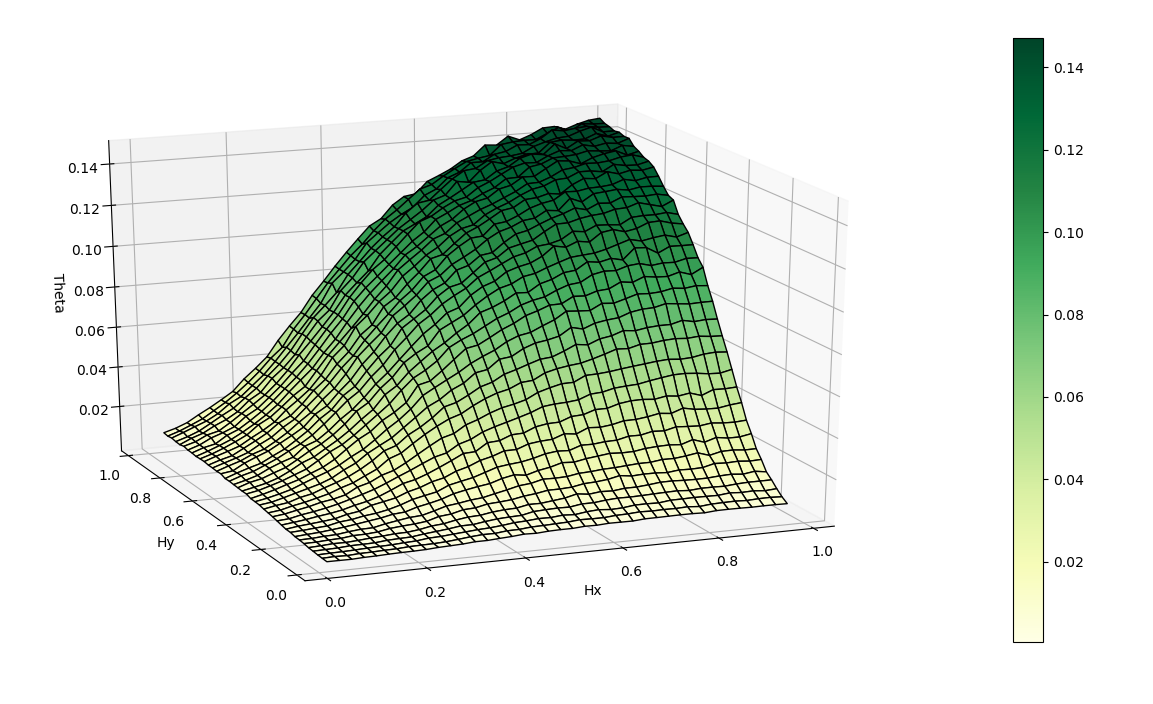}
		\caption{Extremal index vs. Hurst parameters}
		\label{fig:fbmsim}
	\end{figure}
	
	Figure \ref{fig:fbmsim} shows a positive relationship between
$\theta_{\rm b}$ and the Hurst parameters. This can be understood by
noticing that, the smaller is the Hurst parameter, the slower is the
variance increasing, the closer is the Fractional Brownian
motion to the case of a constant variance, i.e. of stationarity. As we
are discussing above, when the Gaussian random field is stationary,
the extremal clusters of the corresponding Brown-Resnick random field
can be very large.  
\end{example}

\section{Acknowledgements}

We are very grateful to Hrvoje Planini\'c for his careful reading of the
paper and pointing out to us the numerous flaws present in a
preliminary version. 

\medskip


\end{document}